\documentclass[11pt,a4paper]{article}

\usepackage[english]{babel}
\usepackage{amsmath}
\usepackage{amsfonts}
\usepackage{amssymb}
\usepackage{amsthm}
\usepackage{float}
\usepackage{graphics}
\usepackage{color}
\usepackage{graphicx}
\usepackage[colorlinks=false]{hyperref}
\usepackage{eso-pic}
\usepackage[T1]{fontenc}
\usepackage[left=2.3cm,right=2.0cm,top=3.5cm,bottom=3.5cm,headsep=1.2cm]{geometry}
\usepackage{fancyhdr}
\usepackage{titlesec}

\usepackage{indentfirst}
\usepackage{enumitem}
\setlist{itemsep=2pt,topsep=4pt}

\usepackage{pgf,tikz}
\usetikzlibrary{arrows}
\usepackage{graphicx}
\usepackage[scriptsize,bf]{caption}
\usepackage{xcolor}
\usepackage{float}
\usepackage{epigraph}

\titleformat{\chapter}[display]
  {\bfseries\small}
  {\filleft\MakeUppercase{\chaptertitlename} \Large\thechapter}
  {1ex}
  {\titlerule\vspace{1ex}\filleft \scshape \bfseries \LARGE}

\def\Z{\mathbb{Z}}

\newcommand{\la}{\lambda}

\def\f{{\varphi}}
\def\im{{\mathtt i}}
\def\g{{\gamma}}

\def\cO{\mathcal{O}}
\def\cC{\mathcal{C}}

\def\N{{\mathbb N}}
\renewcommand{\to}{\rightarrow}

\usepackage{dsfont}

\numberwithin{equation}{section}

\theoremstyle{plain}
\newtheorem{teor}{Theorem}[section]
\newtheorem{ese}[teor]{Example}
\newtheorem{prop}[teor]{Proposition}
\newtheorem{lem}[teor]{Lemma}
\newtheorem{cor}[teor]{Corollary}

\newcommand{\bdm}{\begin{displaymath}}
\newcommand{\edm}{\end{displaymath}}
\newcommand{\bpb}{\begin{prob}}
\newcommand{\epb}{\end{prob}}

\newcommand{\beq}{\begin{equation}}
\newcommand{\eeq}{\end{equation}}
\newcommand{\bem}{\begin{multline}}
\newcommand{\eem}{\end{multline}}
\newcommand{\bes}{\begin{ese}}
\newcommand{\ees}{\end{ese}}
\newcommand{\bde}{\begin{defi}}
\newcommand{\ede}{\end{defi}}
\newcommand{\bpr}{\begin{prop}}
\newcommand{\epr}{\end{prop}}
\newcommand{\ble}{\begin{lem}}
\newcommand{\ele}{\end{lem}}
\newcommand{\bte}{\begin{teor}}
\newcommand{\ete}{\end{teor}}
\newcommand{\bco}{\begin{cor}}
\newcommand{\eco}{\end{cor}}

\theoremstyle{definition}
\newtheorem{defi}[teor]{Definition}
\newtheorem{remark}[teor]{Remark}

\newcommand{\R}{\mathbb{R}}

\newcommand{\T}{\mathbb{T}}

\newcommand{\calO}{{\mathcal O}}

\newcommand{\gots}{{\mathfrak s}}

\newcommand{\al}{\alpha}

\newcommand{\tb}{\mathtt{b}}

\newcommand{{\resonance}}{relevant self-energy cluster }

\newcommand{\ii}{{\rm i}}

\def\cC{{\mathcal C}}

\newcommand{\tk}{{\mathtt k}}

\mathchardef\emptyset="001F

\makeatletter
\newcommand{\myitem}[1][]{
  \protected@edef\@currentlabel{#1}%
\item[#1]
}
\makeatother

\let\origdoublepage\cleardoublepage
\newcommand{\clearemptydoublepage}{%
  \clearpage
  {\pagestyle{empty}\origdoublepage}%
}
\let\cleardoublepage\clearemptydoublepage

\title{\bf{ Reducibility of first order linear operators on tori via Moser's theorem}}

\author{
R. Feola$^*$, 
F. Giuliani$^\dag$ , R. Montalto$^{**}$, M. Procesi$^\dag$
\\
\small
${}^*$ SISSA, Trieste, rfeola@sissa.it
\\
\small
${}^\dag$  Roma Tre, fgiuliani@mat.uniroma3.it, procesi@mat.uniroma3.it
\\
\small
${}^{**}$ Universit\"at Z\"urich, riccardo.montalto@math.uzh.ch
}

\begin{document}

\date{}
\maketitle
\abstract{In this paper we prove  reducibility of classes of linear first order operators on tori by applying  a generalization of  Moser's theorem on straightening of vector fields on a torus. We consider  vector fields which are a $C^\infty$ perturbations of a constant  vector field, and prove that they are conjugated  --by a $C^\infty$ torus diffeomorphism-- to a constant diophantine flow, provided that the perturbation is small in some given $H^{s_1}$ norm and that the initial frequency is in some Cantor-like set.  Actually in the classical results of this type the regularity of the change of coordinates which straightens the perturbed vector field coincides with the class of regularity in which the perturbation is required to be small. This improvement is achieved thanks to ideas and techniques coming from the Nash-Moser theory.

\tableofcontents

\section{Introduction and main results}
In the last years there has been a lot of advances in the study of  KAM theory and almost global existence for classes of quasi-linear and fully nonlinear PDEs on the circle.

In these results, the main issue is to prove the reducibility of quasi-periodically time dependent linear operators.
For instance these operators arise from linearizing such PDEs at small quasi-periodic approximate solutions, whose study  is required by Nash-Moser type schemes, together with a careful quantitative analysis.

%
%

 Given a linear PDE with coefficients which depend on time in a quasi-periodic way, we say that it is reducible if there exists a bounded change of variables  depending quasi-periodically on time (say mapping $H^s\to H^s$ for all times), which makes constant its coefficients. This is a problem which is interesting on itself and has been studied 
for PDEs both on compact and non-compact domains, mostly in a perturbative regime. We mention among others   {\cite{EK2}, \cite{Bam},\cite{BG},\cite{GP},\cite{BGMR},\cite{Mon2}} for the case of linear equations. 
Regarding reducible KAM theory for non linear PDEs the literature is quite vast, we mention  the classical papers \cite{Ku},\cite{W},\cite{KP},\cite{CY} for PDEs on the circle, \cite{GY},\cite{EK},\cite{GYX},\cite{PP},\cite{EGK} for PDEs on $\T^n$.  These works  all deal with PDEs with bounded nonlinearities. Regarding unbounded cases we mention \cite{Ku2}, \cite{LY}, \cite{BBiP2} for semilinear PDEs and 
\cite{Airy},\cite{BBM16},\cite{FP},\cite{Gi},\cite{BM1},\cite{BBHM} for the quasilinear case.
%
\\
In this paper we discuss the reducibility of first-order operators and show that such property  holds, under some explicit  hypotheses, for hyperbolic PDEs on any torus $\T^N$, independently of the spatial dimension. The main idea is to use the identification between derivation operators and vector fields in order to change the PDE reduction problem into a corresponding {\em straightening} problem for a vector field.

We remark that this idea is in spirit very similar to the one used in \cite{BGMR} in order to prove reducibility of a class of quadratic second order  perturbations of the Harmonic oscillator on $\R^N$. Indeed also in that case the key of the proof is that there is an identification between such operators and classical Hamiltonian functions (with $N$-degrees of freedom) so that the reducibility theorem  corresponds to a finite dimensional KAM theorem.

We stress that in order to apply our reducibility result in the context of KAM theory for PDEs we need a very good control on the changes of variables that diagonalize our operators, which corresponds to a very good quantitative control in the  Arnold-Moser  theorem of the norm of the change of variables which straightens the torus. As far as we know such estimates were not known so that our result may have some interest also in that context. 

\medskip

Let us briefly describe our result in the context of ODEs. The  problem of 
 structural stability of linear flows on a torus  has formed the basis of KAM theory.
On the torus $\T^N$ one sets $X_\alpha$ to be the linear vector field with  frequency $\alpha$, namely the vector field which generates the flow $\theta\mapsto \theta+ \alpha t$ with $\alpha\in \R^N$, then the result can be stated as follows: for all  diophantine frequencies $\al$ and for  any appropriately small vector field $f$  there exists a Lipschitz function $\al\to \lambda(\al)$ and a change of variables on $\T^N$ which conjugates $X_{\alpha+\lambda(\alpha)}+f$ to $X_\alpha$.   It is easily seen that $\al\to \al+ \lambda(\al)$ can be extended to  a Lipeomorphism on $\R^N$ so that the statement can be equivalently rephrased as the conjugation between $X_\xi+f$ and $X_{\al(\xi)}$, where the initial frequency $\xi\in \R^N$ is chosen so that the final frequency $\al(\xi)$ is diophantine. This result was first proved by Arnold in \cite{Ar1,Ar2} in the case of analytic vector fields, then by Moser in \cite{Mos67} for finitely differentiable vector fields. The problem was further investigated by R\"ussmann \cite{Rus,Rus2}, P\"oschel \cite{PosTe,Pos1}, Herman \cite{Herm,Herm2}, Salamon \cite{Sala} 
with the purpose of giving optimal bounds on the regularity hypotheses needed on the vector field. See also \cite{Yo}, \cite{FaKr},\cite{Mas}.
%
\\
In this paper, we consider a different approach, namely we consider  $C^\infty$ vector fields $f$ which are small in  some {\em low} Sobolev norm with no further quantitative information on the higher Sobolev norms. Then we prove that the change of variables predicted by Moser's theorem is in fact $C^\infty$ and we give a very good control of the Sobolev norms of this diffeomorphism in terms of the Sobolev norms of $f$. {More precisely we prove that the  diffoemorphism is of the form $\theta\mapsto\theta+h(\theta)$ where,  for all $s$, the $H^s$-norm of $h$ is bounded proportionally to the $H^{s+\mu}$ norm of $f$ for some fixed  $\mu$  {\em independent of $s$}}; this is the content of Theorem \ref{moserMic}. Note that in order to obtain this result we do not rely on approximation by analytic fuctions, as in the aforementioned papers, instead we approach the problem in the spirit of the Nash-Moser theory where one uses interpolation and smoothing estimates in order to control the loss of regularity due to the presence of small divisors. \\
The novelty of our approach is to get a good control of \emph{all} the Sobolev norms of the change of variables by requiring only smallness assumptions on fixed low norm of the perturbation.

\smallskip

As explained before, our more general formulation of the Arnold-Moser theorem  is well suited for applications to the reducibility of linear first-order operators on $H^s(\T^N,\R)$, {which are the main motivation of our work.}

Let us now describe how one attacks reducibility problems in PDEs via Moser's theorem. It is well known that  
a map $f\in C^\infty(\T^N,\R^N)$ can be seen in two ways: as the generator of the flow $\Phi_f(\theta,t)$ given by the ODE $\dot\theta =f(\theta)$, 
  or as the linear functional $L_f$ acting on the space $C^\infty(\T^N,\R)$    as 
$u \mapsto L_f[u](\theta):= f(\theta)\cdot \partial_{\theta} u(\theta)$. This identification is coordinate independent, namely given a diffeomorphism $\theta\to \theta_1=\Psi(\theta)$, then the derivation operator associated to the  push-forward of $f$  is the conjugate of the derivation operator $L_f$ with respect to the $\Psi$,  in formulas
$$
L_{\Psi_* f} u = \Psi^{-1} \circ (L_f [\Psi\circ u] )\,.
$$ 
Now clearly the same torus diffeomorphism which straightens  the vector field  puts the corresponding derivation operator to constant coefficients.
\\
The idea of analyzing a first order linear differential operator through its associated vector field is the so called method of characteristics, which is the classical way in which a first order linear PDE is reduced to a (possibly non-linear) ODE. In this case it allows us to solve linear non-homogeneous problems of the form
\[
L_{\xi+f}u= (\xi+ f(\theta))\cdot \partial_\theta u(\theta) = g(\theta)+\mathtt c
\]
where $\xi\in \R^N$ are parameters, $f\in C^\infty(\T^N,\R^N), g\in C^\infty(\T^N,\R)$  are given, $f$ is small in some fixed $s_1$-Sobolev norm and $u\in H^{s}(\T^N,\R),\mathtt c\in \R$ are the unknowns. By Arnold-Moser theorem for all diophantine frequencies $\al(\xi)$ there exists a change of variables which conjugates  $L_{\xi+f}$ to the constant coefficients operator $L_{\al(\xi)}$, we have
\[
L_{\al(\xi)} v = \alpha(\xi)\cdot \partial_{\theta} v = \tilde g + \mathtt c\,,\quad v = \Psi^{-1}\circ u \,,\quad \tilde g= \Psi^{-1}\circ g
\]
and this equation is trivially solved by setting $\mathtt c$ to be the average over $\T^N$ of $\tilde g$. Going back to the original variables we get quantitative tame  estimates on the Sobolev norm of the solution.
\\
Following this line of thought we use Theorem \ref{moserMic}   to  prove the reducibility of a linear transport equation of the form
\begin{equation}\label{sistema}
\dot u= L(u)
\end{equation}
where $L$,  defined as
$$
u\to L u:=  \sum_{i=1}^d (\zeta_i+ a_i(\omega t,x))\partial_{x_i} u\,,
$$
depends on time in a quasi-periodic way.
Here $\xi:=(\omega,\zeta)\in \R^\nu\times\R^d$ are parameters to be fixed and $x\in \T^d$.
\\
The reducibility of such systems  amounts to proving that there exists a linear  time dependent operator  $\Psi(\omega t)$ acting on $C^\infty(\T^d,\R)$ which reduces $L$ to constant coefficients, i.e. there exists a map $ \R^\nu \times \R^d \mapsto \R^d$, $\zeta \mapsto m(\omega, \zeta)$ such that 
\begin{equation}\label{suca}
- \Psi^{-1}\Psi_t +\Psi^{-1}((\zeta + a(\omega t, x))\cdot \partial_{x} \Psi )\cdot \partial_x = m(\omega, \zeta) \cdot \partial_x\,.
\end{equation}
We look for an operator $\Psi(\omega t)$ of the special form
\[
\Psi(\omega t) v := v\circ \Psi(\omega t)\,,\quad \Psi(\omega t, x)= x + \beta(\omega t,x).
\] 
Here with an abuse of notation we are representing with the same symbol $\Psi$ the time dependent torus diffeomorphism and its action on functions.
\\
In this way our reduction equation \eqref{suca} becomes
\begin{equation}\label{su2}
\Psi^{-1}(\omega\cdot\partial_\f \beta + (\zeta + a(\f, x))\cdot(1+\partial_x \beta) ) = m_\infty(\omega, \zeta)
\end{equation}
with unknowns $m_\infty\in \R^d$ and $\beta\in C^\infty(\T^\nu\times\T^d,\R^d)$.
\\
 One can directly see  that solving equation \ref{su2} is equivalent to finding a change of variables of the form
 $\Psi:(\f,x)\to (\f,x+\beta(\f,x))$   which straightens  the vector field
\begin{equation}\label{befana}
X:= \omega\cdot\frac{\partial}{\partial \f} + (\zeta + a(\f, x))\cdot\frac{\partial}{\partial x}\,,\qquad \Psi_* X= \omega\cdot\frac{\partial}{\partial \f} + m_\infty(\omega,\zeta)\cdot\frac{\partial}{\partial x}.
\end{equation}
In order to do this we just need a reformulation of Theorem \ref{moserMic}, which we give in Proposition \ref{esempio1}. This allows us to solve \eqref{sistema} 
for $(\omega,\zeta)$ in appropriate Cantor sets, and to prove that in such case there is no growth of the Sobolev norms, see Theorem \ref{riducibilita trasporto}.
\\
A further application is in the case of a PDE as
\[
u_t + (m_0(\omega) +a(\omega t,x))u_x = g(t,x)\,,
\]
hence in the same class as the previous example but with $x\in \T$ and with {\em less parameters}, i.e. here $\zeta=m_0(\omega)$. Of course the straightening Theorem can be reformulated to fit this case, see Corollary \ref{cor m 0 omega a 0 omega}. Then in Corollary \ref{misura 1-d} we show that reduce to 
\[
v_t + m_\infty(\omega) v_x = \tilde g(t,x) 
\]
for all the $\omega$ such that $(\omega,m_\infty(\omega))$ is diophantine. Finally we show that  such set has positive measure,  under some weak hypotheses on $m_0(\omega)$.

\smallskip

Let us now briefly discuss the strategy of proof of Theorem \ref{moserMic}.
 We construct the transformation $\Psi^{(\infty)}$ by means of an iterative KAM-type scheme in Sobolev class.  Our method is based on constructing a sequence of transformations $(\Phi_n)_{n \in \N}$ of the form $\Phi_n (\theta) = \theta+g_n(\theta)$, $n \in \N$ which reduce quadratically the size of the vector field $f$. More precisely, at the $n$-th step of our procedure, we deal with an operator \[X_n := (\al_n(\xi) +f_n(\xi,\theta))\cdot \frac{\partial}{\partial\theta}.\] Then, we choose a function $g_n(\xi,\theta)$ so that 
\begin{equation}\label{omologica alpha introduzione}
\alpha_n \cdot \partial_\theta g_n + \Pi_{N_n}f_n(\theta) = \langle f_n \rangle_{\T^{N}}
\end{equation}
where $N_n = N_0^{(\frac32)^n}$, $\Pi_{N_n}$ is the orthogonal projection on the Fourier modes $|k| \leq N_n$ and $\langle f_n \rangle_{\T^{N}}$ is the average of the vector valued function $f_n$ w.r.t. to $\theta\in\T^N$. The equation \eqref{omologica alpha introduzione} can be solved by imposing the {\it diophantine  conditions} 
$$
|\alpha_n\cdot k| \geq \frac{\gamma}{\langle k \rangle^\tau}\,, \qquad \forall 0 < |k| \leq N_n ,
$$
for some $\gamma \in (0, 1)$ and $\tau > \nu$. Then, setting $\Phi_n (\theta) = \theta+ g_n(\theta)$, one gets 
$$
X_{n + 1} = (\Phi_n)_* X_n =  (\alpha_{n + 1}(\xi)+ f_{n+1})\cdot \frac{\partial}{\partial\theta} ,
$$
where $\alpha_{n + 1} := \alpha_n + \langle f_n \rangle_{\T^{N}}$, $|f_{n + 1}| \simeq |\Pi_{N_n}^\bot f_n| + |f_n|^2$.

We show that $f_n$ converges to $0$ in any Sobolev space $H^s$, $s \geq 0$ provided $\| f_0 \|_{H^{s_1}}$ is small for some fixed Sobolev index $s_1$ (which has to be taken large enough). Then we define the change of variables $\Psi^{(n)} := \Phi_0 \circ \Phi_1 \circ \ldots \circ \Phi_n$ and show that it is of the form $\Psi^{(n)}(\theta)= \theta + h_n(\theta)$. Moreover the sequence $h_n\in C^\infty(\T^N,\R^N)$ converges to a function $h_\infty=h$   in any $H^s$, $s \geq 0$. 
Then we verify that $\Psi(\theta)= \theta+h(\theta)$ is a torus diffeomorphism which conjugates $X_0$ to $\alpha_{\infty}(\xi)\cdot \frac{\partial}{\partial\theta} $.
The final step is to show that the condition $\al_\infty(\xi)$ diophantine is sufficient to ensure that all the Melnikov conditions are satisfied.

  The results obtained in Theorem \ref{riducibilita trasporto} and Corollary \ref{misura 1-d} will be applied in \cite{FGP} and in \cite{Montsublin}, since they allow to reduce to constant coefficients the highest order term of some linear equations. Note that in view of the application to \cite{FGP}, it is needed also to study the dependence of the transformation $\Psi^{(\infty)}$ w.r. t. an additional parameter $\lambda$ which is in a Banach space. This is done in Corollary \ref{cor1}. 

\medskip

As it has been explained above, the reducibility of linear equations is the key step for proving existence and stability of quasi-periodic solutions for nonlinear PDEs. These problems are much harder when the nonlinear part of the vector field of the equation contains derivatives of the same order as the linear part (fully nonlinear case). Indeed, the usual KAM algorithms fail since the {\it loss of derivatives} accumulates quadratically along the iterative scheme, implying that such a scheme does not converge. Note that the operator $L$ is a perturbation of order $1$ ($f_0 \cdot \partial_\theta$) of a diagonal operator with constant coefficients of the same order ($\xi\cdot\partial_{\theta}$).\\
In one-space dimension, Baldi-Berti-Montalto \cite{Airy} proposed a method for proving the reducibility of linear operators of this form. Such a procedure is split in two steps: a regularization procedure, in which the operator is reduced  to a diagonal one plus a bounded remainder, and a KAM scheme which completely diagonalizes the linear operator by reducing {\it quadratically } the size of the remainders.\\
The reduction procedure that one has to perform depends strongly on the linear dispersion law of the PDE. 
For instance when the dispersion law is superlinear, as in KdV or NLS,  in order to reduce to constant coefficients the leading term one applies a torus diffeomorphism of the form $x\to x+ \beta(x,\omega t)$ where $\beta$ solves 
\begin{equation}\label{nome}
(1+a(\f,x))(1+\beta_x)= m(\f)\,,
\end{equation}
which can be solved directly by integration.
\\
Comparing this equation with \eqref{su2} one sees that here $\partial_{\f}$ does not appear and the right hand side still depends on $\f$. This is due to the fact that the leading derivative is the spatial one.  
\\
When the dispersion law is linear, as in the case we deal with, there are some difficulties. The time and the space play the same role and one has to deal with transport-like equations as \eqref{su2}.\\
One could solve these equations by perturbative arguments. As we explained above, the classical KAM reducibility scheme fails. Our strategy is to construct, at the $n$-th step of the iteration, a change of variables $\Phi_n$ of the form $\Phi_n u = u(\theta+g_n(\theta))$. {The main point} is that the sequence of transformations $\Psi^{(n)} := \Phi_0 \circ \Phi_1 \circ \ldots \circ \Phi_n$ does not converge w.r.t.  the operator norm of bounded linear operators $H^s \to H^s$, since $\Phi_n - {\rm Id} = O(g_n \cdot\partial_{\theta})$ is small in size but unbounded of order one. This is the reason for approaching the problem from the point of view of the straightening of the corresponding vector field.

\smallskip

We remark also that in \cite{BBHM}  a transport equation similar to equation \eqref{su2} appears in the reducibility of the linearized operator. In this case the corresponding vector field is more degenerate, in the sense that, in the basis $\partial/\partial \f$, $\partial/\partial x$, has the form $(\omega, O(\varepsilon))$, instead of $(\omega, \zeta+O(\varepsilon))$ as in \eqref{befana}. In \cite{BBHM} the authors reduce the problem to the study of a nonlinear ODE and then they apply a Nash-Moser Hormander Theorem, see \cite{BaHa},  to solve it. In our case the same strategy fails.

 
%

\smallskip

The paper is organized as follows. In Section \ref{sezione functional setting}, we introduce the functional setting needed to state precisely Theorem \ref{moserMic}. In Section \ref{tiraddrizzotutto}, we state our main results at the level of vector fields. Then in Section \ref{PDE} we state and prove the corresponding results on reducibility of PDEs. In Section \ref{quattro} we prove Theorem \ref{moserMic}, which is based on the iterative Lemma \ref{ittero}. The inductive step of such a Lemma is basically proved in Lemma \ref{lemma KAM step}. Finally, in the Appendix \ref{lemmitecnici}, we collect some technical lemmata concerning Sobolev spaces and operators induced by diffeomorphisms of the torus which are of importance for our proofs. 
 
 \thanks{{\bf Acknowledgements:} Two of the authors were supported by ERC grant 306414 HamPDEs under FP7. The authors wish to thank P. Baldi, L. Biasco  E. Haus and J. Massetti for helpful discussions.}
\section{Functional setting}\label{sezione functional setting}
\paragraph{Sobolev functions.} Given $N, m \in \N$, we consider  real valued functions $u(\theta)\in L^2(\T^{N}, \mathbb{R}^m)$:
\[
u(\theta)= \sum_{k\in \Z^N}u_{k}e^{\mathrm i k\cdot \theta}\,,\quad \overline u_k =u_{-k}.
\]
We use the simplified notation $L^2$ to denote $L^2(\T^{N},\mathbb R^m)$. 
We define the Sobolev space 
\begin{equation}\label{space1}
H^s(\T^{N}, \mathbb{R}^m):=\left\{ u(\varphi, x)\in L^2 : \lVert u \rVert_s^2:=\sum_{k\in\mathbb{Z}^N} \langle k \rangle^{2 s} \lvert u_{k} \rvert^2 <\infty \right\}
\end{equation}
where $\langle k \rangle:=\max( 1,|k|)$.

  If we separate the variables $\theta= (\varphi,x) \in \T^{\nu+d}$, we may consider a  real valued functions $u(\varphi,x)\in L^2$ 
 as a $\varphi$-dependent family of functions $u(\varphi, \cdot)\in L^2(\T^d_x, \mathbb{R}^m)$ with the Fourier series expansion
\begin{equation*}
u(\varphi, x)=\sum_{j\in\mathbb{Z}^d } u_{j}(\varphi)\,e^{\mathrm{i} j \cdot  x}=\sum_{\ell\in\mathbb{Z}^{\nu}, j \in \mathbb{Z}^d } u_{\ell, j} \,e^{\mathrm{i}(\ell \cdot \varphi+j \cdot x)}.
\end{equation*}
In this case it may be more convenient to describe the  Sobolev space 
\begin{equation}\label{space}
H^s(\T^{\nu+d}, \mathbb{R}^m):=\left\{ u(\varphi, x)\in L^2 : \lVert u \rVert_s^2:=\sum_{\ell\in\mathbb{Z}^{\nu}, j\in\mathbb{Z}^d} \langle\ell, j \rangle^{2 s} \lvert u_{\ell j} \rvert^2 <\infty \right\}
\end{equation}
where $\langle \ell, j \rangle:=\max\{ 1, \lvert \ell \rvert, \lvert j \rvert\}$, $\lvert \ell \rvert:=\sum_{i=1}^{\nu} \lvert \ell_i \rvert$, $|j| := \sum_{i = 1}^d |j_i|$. 
If $\gots_0:=[N/2] + 1$ then for $s\geq \gots_0$ the spaces $H^s:=H^s (\T^{N},\mathbb R^m)$ are embedded in $ L^{\infty}(\T^{N},\mathbb R^m)$ and they have the algebra and interpolation structure, namely  $\forall u, v\in H^s$ with $s\geq \gots_0$:
\[
\lVert u v \rVert_s\le C(s,N) \lVert u \rVert_s \lVert v \rVert_s, 
\]
\begin{equation}\label{TameProduct}
\lVert u\,v \rVert_s\le C(N)\,\lVert u \rVert_s\lVert v \rVert_{\gots_0}+C(s,N) \lVert u \rVert_{\gots_0}\lVert v \rVert_s. 
\end{equation}

Here $C(N),C(s,N)$ are positive constants independent of $u,v$. The above estimates are classical and one can see for instance the Appendix of \cite{BBoP} for the proof.

\paragraph{Lipschitz norm.} Fix $N\in\mathbb{N}$, $N\geq 1$ and let $\calO$ be a compact subset of $\mathbb{R}^{N}$. 
For a function $u\colon \calO\to E$, where $(E, \lVert \cdot \rVert_E)$ is a Banach space, we define the sup-norm and the lip-seminorm of $u$ as
\begin{equation}\label{suplip}
\begin{aligned}
&\lVert u \rVert_E^{\sup}:=\lVert u \rVert_{E}^{\sup, \calO}:=\sup_{\xi\in\calO} \lVert u(\xi) \rVert_E,\\
&\lVert u \rVert_{E}^{lip}:=\lVert u \rVert_{E}^{lip, \calO}:=\sup_{\substack{\xi_1, \xi_2\in \calO,\\ \xi_1\neq \xi_2}} \frac{\lVert u(\xi_1)-u(\xi_2)\rVert_E}{\lvert \xi_1-\xi_2\rvert}.
\end{aligned}
\end{equation}
Fix $\gamma>0$.
We will use the following Lipschitz norms 
\begin{equation}\label{tazza10}
\lVert u \rVert_s^{\g, \calO}:=\lVert u \rVert_s^{sup, \calO}+\gamma \lVert u \rVert_{s-1}^{lip, \calO}, \quad u\in H^s, \quad \forall s\geq [N/2]+ 3,
\end{equation}
\begin{equation}\label{tazzone}
\lvert m \rvert^{\gamma, \calO}:=\lvert m \rvert^{sup, \calO}+\gamma \lvert m \rvert^{lip, \calO}, \quad m\in\mathbb{R}.
\end{equation}

\paragraph{Diffeomorphisms of the torus.} We  consider diffeomorphisms of the $N$-dimensional torus
\begin{equation}\label{lenovo}
\Phi\colon \T^{N}\to \T^{N}, \qquad   \Phi\colon \theta  \mapsto \theta + h(\theta)= \widehat{\theta}
\end{equation}
where $h : \T^N  \to \R^N$ is some $C^{\infty}$ function with $\|\beta\|_{\gots_0+1}\le \frac12 $.   We denote the inverse diffeomorphism  as
\[
\Phi^{-1}\colon \T^{N}\to \T^{N}, \qquad   \Phi^{-1}\colon \theta  \mapsto \theta + \tilde h(\theta)
\]
with $\tilde{h}$ a $C^{\infty}$ function. 
 With an abuse of notation we  identify transfomations like \eqref{lenovo} with the {corrsponding linear} operators acting on  $H^s(\T^{N})$  as
\begin{equation}\label{grotta3}
\Phi\colon H^s\to H^s, \qquad\Phi h(\theta)=h(\theta +h(\theta)).
\end{equation}
Similarly we consider the action of $\Phi$ on the vector fields on $\T^{N}$ by the pushforward.  Explicitly
we denote by $ \mbox{T}\, (\T^{N})$ the tangent space of  $\T^{N}$.\\ Now given a vector field $X\colon \T^{N}\to \mbox{T}\, (\T^{N})$
\begin{equation}\label{grotta4}
\begin{aligned}
X(\theta)
& = \sum_{j=1}^{N} X_j(\theta)\,\frac{\partial}{\partial \theta_j}\,,\quad X_1, \ldots, X_N \in C^{\infty}(\T^{N},\R)
\end{aligned}
\end{equation}
its pushforward is
\[
(\Phi_*X)(\theta)=d \Phi (\Phi^{-1}(\theta)) [X (\Phi^{-1}(\theta))]= \sum_{i=1}^N\Phi^{-1}\left( X_i +\sum_{j=1}^N \frac{\partial{X_i}}{\partial \theta_j} X_j\right)\frac{\partial}{\partial \widehat \theta_i}.
\]

\medskip

We refer to the Appendix \ref{lemmitecnici} for technical lemmata on the tameness properties of the Lipschitz, Sobolev norms and bounds for the diffeomorphisms of the torus. 

\medskip

\paragraph{Reversible vector fields.}
Let $S : \T^N \to \T^N$ be the involution $\theta \mapsto - \theta$. We say that a vector field $ X: \T^N \to T(\T^N)$ is {\it reversible} if 
$$
X \circ S = - S \circ X\,.
$$
This is equivalent to say that $X$ is even w.r. to the variable $\theta$. A diffeomorphism of the torus $\Phi : \T^N \to \T^N$ is said to be {\it reversibility preserving} if 
$$
\Phi \circ S = S \circ \Phi\,.
$$
The above definition is equivalent to require that $\Phi$ is odd w.r. to the variable $\theta$. It is a straightforward calculation to verify that if $X$ is reversible and $\Phi$ is reversibility preserving, then the push-forward $\Phi_* X$ is still reversible. 

\paragraph{Linear operators}
A $C^\infty$ vector field $X(\theta) = \sum_{j = 1}^N X_j(\theta) \frac{\partial}{\partial \theta_j}$ induces a linear operator acting on then space of functions $u : \T^N \to \R$, that we denote by $X(\theta) \cdot \partial_\theta = \sum_{j = 1}^N X_j(\theta) \partial_{\theta_j}$. More precisely, the action of such a linear operator is given by 
$$
H^{s}(\T^N, \R) \to H^{s - 1}(\T^N, \R)\,, \quad u(\theta) \mapsto X(\theta) \cdot \partial_\theta u(\theta).
$$

\section{A tame version of Moser's  theorem}\label{tiraddrizzotutto}

Our first result, contained in Theorem \ref{moserMic}, is a revisited version of Moser's theorem on weakly perturbed vector fields on the torus. The main difference relies on the type of estimates that are provided for the changes of coordinates that one performs in order to reduce the perturbed vector field to constant coefficients. These estimates are \emph{tame} in the sense that the highest norms appear linearly. This kinds of bounds allow to construct iteratively a  smooth change of coordinates by requiring smallness conditions only on a fixed {\it low} Sobolev norm of the first perturbation $f_0$. 

\medskip

Let us fix $N\in\mathbb{N}$, $N\geq 1$, $\cO_0$ a compact subset of $\mathbb{R}^{N}$ with positive Lebesgue measure and consider
\begin{equation}\label{tripoli}
\tau:=N +2, \quad s_0\geq [N/2]+ 3, \quad \gamma\in (0, 1).
\end{equation}
\paragraph{Notations.}
We denote for a vector valued function $u(\theta)$ its average as 
\[
\langle u \rangle:=\frac{1}{(2 \pi)^{N}}\int_{\T^{N}} u\,d\theta.
\]
We  denote any constant depending only on $N,\cO_0$ as $C$ and correspondingly we say $a\lesssim b$
if $a\le C b$. A constant depending on parameters $p$ is denoted by $C_p$ and as above  we say  $a\lesssim_p b$
if $a\le C_p b$.

\begin{teor}[Tame Moser Theorem]\label{moserMic}
	Consider for $\xi  \in \cO_0  \subseteq \R^{N}$ a Lipschitz family of vector fields  on $\T^{N}$
	\begin{equation}\label{mao1}
	X_0:=	(\xi   + f_0(\theta;\xi)) \cdot \frac{\partial}{\partial \theta}
	\end{equation}
	\begin{equation}
	\quad f_0(\cdot;\xi)\in H^s(\T^{N},\R^N)\quad  \forall s\geq s_0.
	\end{equation}
	There exists $s_1 \geq s_0+2\tau+4$,  and $\eta_\star=\eta_\star(s_1)>0$ such that if
	\begin{equation}\label{picci}
	\g^{-1} \|f_0\|^{\g,\mathcal O_0}_{s_1} :=\delta \le \eta_\star 
	\end{equation}
	then there exists a Lipschitz  function $\alpha_\infty : \cO_0 \to \R^N, \xi \mapsto \alpha(\xi)$ with  
	\begin{equation}\label{tordo4}
	\qquad \lvert \alpha_\infty- \rm{ Id}_{\R^N} \rvert^{\gamma,\cO_0}\leq \gamma \delta,\,
	\end{equation}
	such that in the set
	\begin{equation}\label{buoneacque}
	\cO_{\infty}^{2 \gamma}:= \Big\{\xi \in \cO_0: \; |\alpha_\infty(\xi)\cdot k |>\frac{2\g }{\langle k\rangle^{\tau}}\,,\;\forall k \in \Z^N \setminus \{0\}\Big\}
	\end{equation}
	the following holds. There exists a map
	\begin{equation}\label{perunpugnodidollari}
	\beta: \cO_{\infty}^{2 \gamma}\times \T^{N}\to \R^N \,,\quad  \|\beta \|^{\g,\cO_{\infty}^{2 \gamma}}_{s} \lesssim_s  \g^{-1}\|f\|^{\g,\cO_0}_{s+2\tau+4},
	\quad \forall s\geq s_0,\, _{n+1}{s\in\mathbb{N}}
	\end{equation}
	so that $\Psi: \theta \mapsto \theta+\beta(\theta)=\widehat \theta$ is a diffeomorphism of $\T^{N}$ and for all $\xi \in \cO_{\infty}^{2 \gamma}$
	\begin{equation}\label{tordo6}
	\begin{aligned}
	\Psi_* X_0 & :=  (\Psi)^{-1}\big(\xi + f_0 + (\xi + f_0)\cdot \partial_\theta \beta \big) \cdot \frac{\partial}{\partial \widehat\theta} = \alpha_\infty(\xi)\cdot \frac{\partial}{\partial \widehat\theta}\; .
	\end{aligned}
	\end{equation}
	Furthermore, if $f_0$ is a reversible vector field (i.e. $f_0 = {\rm even}(\theta)$), then the diffeomorphism $\theta \mapsto \theta + \beta(\theta)$ is reversibility preserving (i.e. the function $\beta = {\rm odd}(\theta)$). \\
{	Finally, the Lebesgue measure of the set ${\cal O}_0 \setminus {\cal O}_\infty^{2 \gamma}$ satisfies the bound 
	\begin{equation}\label{stima di misura 0}
	|{\cal O}_0 \setminus {\cal O}_\infty^{2 \gamma}| \lesssim \gamma\,. 
	\end{equation}}
\end{teor}
\begin{remark}
	Note that condition \eqref{tordo4} implies that the map $\xi\mapsto \alpha(\xi)$ is a lipeomorphism and 
	$|\xi(\al)|^{lip}\le 2$. Hence the estimate of the measure of the complementary set $\cO_0\setminus\cO^{2\g}_\infty$ is trivial.
\end{remark}

\begin{cor}{(Parameter dependence)}\label{cor1}
	Under the assumptions of Theorem \ref{moserMic}, suppose that $f_0=f_0(\theta, \lambda; \xi )$ depends on some parameter $\lambda\in B_{E}$, where $B_E$ is a ball centered at the origin of a Banach space $E$, and denote with $\Delta_{12} g(\lambda):=g(\lambda_1)-g(\lambda_2)$ for some $\lambda_1, \lambda_2\in B_E$. Then the frequency vector $\alpha_\infty=\alpha_{\infty}(\xi, \lambda)$, 
	given by Theorem \ref{moserMic} applied to $f_{0}(\theta,\lambda;\xi)$, satisfies 
	the following:
	\begin{itemize}
		\item[$(i)$] recalling \eqref{buoneacque}, for $\xi \in \calO^{2\gamma}_{\infty}(\lambda_1)\cap \calO^{2\gamma}_{\infty}(\lambda_2)$
		\begin{equation}
		\lvert \Delta_{12} \alpha_{\infty} \rvert\le 2 \lvert \Delta_{12} \langle f_0 \rangle \rvert.
		\end{equation}
		\item[$(ii)$] There exists a  positive constant $\tb$ satisfying $s_0+\tb < s_1$,  such that we have the following estimate for $\xi  \in \calO^{2\gamma}_{\infty}(\lambda_1)\cap \calO^{2\gamma}_{\infty}(\lambda_2)$:
		\begin{equation}
		\lVert \Delta_{12} \beta   \rVert_{s_0-1}\lesssim_{s_1} \gamma^{-1} \lVert \Delta_{12} f_0   \rVert_{s_0+\tb}.
		\end{equation}

	\end{itemize}

\end{cor}
As explained in the introduction we now divide the variables $\theta$ in time and space, $\theta=(\varphi,x)\in \T^{\nu+d}$. Similarly we write $\xi=(\omega,\zeta)\in \R^{\nu+d}$. We have the following result
\begin{prop}\label{esempio1}
Consider for $(\omega,\zeta)  \in \cO_0  \subseteq \R^{\nu+d}$ a Lipschitz family of vector fields  on $\T^{\nu+d}$
\begin{equation}\label{mao1bbb
}
X_0:=	\omega\cdot \frac{\partial}{\partial \varphi}  +( \zeta+ a_0(\varphi,x;\omega,\zeta)) \cdot \frac{\partial}{\partial  x}
\end{equation}
\begin{equation}\label{a 0 Hs}
\quad a_0(\cdot;\omega,\zeta)\in H^s(\T^{\nu+d},\R^{d})\quad  \forall s\geq s_0.
\end{equation}
Fix $s_1,\eta_\star$ as in Theorem \ref{moserMic},  if
\begin{equation}\label{piccib}
\g^{-1} \|a_0\|^{\g,\mathcal O_0}_{s_1} :=\delta \le \eta_\star 
\end{equation}
then there exists a Lipschitz  function $m : \cO_0 \to \R^{d}, (\omega,\zeta) \mapsto m_\infty(\omega,\zeta)$ such that denoting $\alpha_\infty(\omega,\zeta)=(\omega,m_\infty(\omega,\zeta))$ we have
\begin{equation}\label{tordo4b}
\qquad \lvert  \alpha_\infty- \rm{ Id}_{\R^{\nu+d}}  \rvert^{\gamma}\leq \gamma \delta, \,
\end{equation}
such that in the set
\begin{equation}\label{buoneacqueb}
\cO_{\infty}^{2 \gamma}:= \Big\{(\omega,\zeta) \in \cO_0: \; |\omega \cdot \ell +m_\infty(\omega,\zeta)\cdot j  |>\frac{2\g }{\langle \ell,j\rangle^{\tau}}\,,\;\forall (\ell,j)\in \Z^{\nu+d} \setminus \{0\}\Big\}
\end{equation}
the following holds. There exists a map
\begin{equation}\label{perunpugnodidollarib}
\beta:  \T^{\nu+d}\times \cO_{\infty}^{2 \gamma}\to \R^d \,,\quad  \|\beta \|^{\g,\cO_{\infty}^{2 \gamma}}_{s} \lesssim_s  \g^{-1}\|a_0\|^{\g,\cO_0}_{s+2\tau+4},
\quad \forall s\geq s_0,\,  {s\in\mathbb{N}}
\end{equation}
so that $\Psi: (\varphi,x) \mapsto (\varphi,x+\beta(\varphi,x;\omega,\zeta))= (\varphi,\widehat x)$ is a diffeomorphism of $\T^{\nu+d}$ and for all $(\omega,\zeta) \in \cO_{\infty}^{2 \gamma}$
\begin{equation}\label{tordo6b}
\begin{aligned}
\Psi_* X_0 & := \omega\cdot \frac{\partial}{\partial \varphi}  + \Psi^{-1}\big(\omega\cdot \partial_{\varphi}\beta + \zeta + a_0 +  (\zeta + a_0)  \cdot \partial_x \beta \big) \cdot \frac{\partial}{\partial \widehat x} \\
& = \omega\cdot \frac{\partial}{\partial \varphi}  + m_\infty (\omega, \zeta) \cdot \frac{\partial}{\partial \widehat x} .
\end{aligned}
\end{equation}
Furthermore if $a_0$ is a reversible vector field (i.e. $a_0 = {\rm even}(\varphi, x)$) then the diffeomorphism $(\varphi, x) \mapsto \big(\varphi, x + \beta(\varphi, x) \big)$ is reversibility preserving (i.e. $\beta = {\rm odd}(\varphi, x)$). 
\end{prop}
\begin{remark}
In the previous proposition we are stating that if $f_0$ in Theorem \ref{moserMic} has the form $f_0=(0,\dots,0,a_0)$ then also the frequency $\al_\infty$ and the change of variables $\beta$ have the same form.
\end{remark}
We now wish to consider the case where $\zeta$ is not an independent parameter but a given function of $\omega$ . We have the following
\begin{cor}\label{cor m 0 omega a 0 omega}
	Consider for $\omega  \in \Omega_0  \subseteq \R^{\nu}$ a Lipschitz family of vector fields  on $\T^{\nu+d}$
	\begin{equation}\label{mao1b}
	X_0:=	\omega\cdot \frac{\partial}{\partial \varphi}  +( m_0{(\omega)}+ a_0(\varphi,x;\omega)) \cdot \frac{\partial}{\partial  x}.
	\end{equation}
Here $m_0(\omega)$ is a Lipschitz function and
	\begin{equation}\label{pasta col pomodoro}
	\quad a_0(\cdot;\omega)\in H^s(\T^{\nu+d},\R^{d})\quad  \forall s\geq s_0.
	\end{equation}
	Fix $s_1,\eta_\star$ as in Theorem \ref{moserMic},   if
	\begin{equation}\label{piccic}
	\g^{-1} \|a_0\|^{\g,\Omega_0}_{s_1} :=\delta \le \eta_\star  
	\end{equation}
	then  in the set
	\begin{equation}\label{buoneacquec}
	\Omega_{\infty}^{2 \gamma}:= \Big\{\omega \in \Omega_0: \; |\omega \cdot \ell + m_\infty(\omega,m_0(\omega))\cdot j  |>\frac{2\g }{\langle \ell,j\rangle^{\tau}}\,,\;\forall (\ell,j)\in \Z^{\nu+d} \setminus \{0\}\Big\}
	\end{equation}
	the map $\beta $ of Proposition \ref{esempio1}, restricted to $\zeta=m_0(\beta)$ diagonalizes $X_0$ as in formula \eqref{tordo6b}. Moreover	\begin{equation}\label{prrr}
	\begin{aligned}
	&  \|\beta \|^{\g,\Omega_{\infty}^{2 \gamma}}_{s} \lesssim_s  (1+|m_0|^{lip,\Omega_0})\g^{-1}\|a_0\|^{\g,\Omega_0}_{s+2\tau+4}\,, \\
	& |m_0 - m_\infty|^\gamma \lesssim \| a_0\|_{s_1}^{\gamma, \Omega_0}
	 \end{aligned}
	\end{equation}
	Furthermore if $a_0$ is a reversible vector field (i.e. $a_0 = {\rm even}(\varphi, x)$) then the diffeomorphism $(\varphi, x) \mapsto \big(\varphi, x + \beta(\varphi, x) \big)$ is reversibility preserving (i.e. $\beta = {\rm odd}(\varphi, x)$). 
\end{cor}
\begin{proof}
	We wish to apply Proposition \eqref{esempio1}, we consider the map $M:\omega\mapsto (\omega,m_0(\omega))$ and denote by $\cO_0$ the image of $\Omega_0$ through this map. Then we consider $\widetilde a_0(\f,x;\omega,\zeta) = a_0(\f,x;\omega)$  in this way  the dependence on $\zeta$ is trivial and we have
	\[
	\g^{-1}\|\widetilde a_0\|_{s_1}^{\g,\cO_0}= \g^{-1}\| a_0\|_{s_1}^{\g,\Omega_0} := \delta \le \eta_\star.
	\]
We thus apply Proposition \eqref{esempio1} to the vector field
\[
\widetilde X_0:=	\omega\cdot \frac{\partial}{\partial \varphi}  +(\zeta+ \widetilde a_0(\varphi,x;\omega,\zeta)) \cdot \frac{\partial}{\partial  x}.
\]
We produce a change of variables $x\mapsto x+ \widetilde\beta(\f,x;\omega,\zeta)$ which diagonalizes $\widetilde X_0$ in the set $\cO_{\infty}^{2\g}$. We now restrict our parameter set to $\zeta= m_0(\omega)$. By definition $\omega\in \Omega_0\Leftrightarrow (\omega,m_0(\omega))\in \cO_0$, so the restriction of $\cO_{\infty}^{2\g}$ to $\zeta= m_0(\omega)$ is $\Omega_{\infty}^{2\g}$.
It remains to prove the bound \eqref{prrr}. Setting $ \beta(\f,x,\omega)= \widetilde\beta(\f,x;\omega,m_0(\omega)) $ we have
\[
\|\beta \|^{sup,\Omega_{\infty}^{2 \gamma}}_{s}= \|\widetilde\beta \|^{sup,\cO_{\infty}^{2 \gamma}}_{s}\,,\quad \|\beta \|^{lip,\Omega_{\infty}^{2 \gamma}}_{s}\le  \|\widetilde\beta \|^{lip,\cO_{\infty}^{2 \gamma}}_{s}(1+|m_0|^{lip,\Omega_0}) 
\]
hence the result follows.
\end{proof}
\section{Reducibility}\label{PDE}
In this section we state some reducibility results of quasi-periodic transport type equations. They are obtained as a consequence of Theorem \ref{moserMic}. 
\begin{teor}\label{riducibilita trasporto}
Consider the transport equation
\begin{equation}\label{eq traporto dimensione alta}
\partial_t u +  \big( \zeta + a_0(\omega t, x; \omega, \zeta) \big) \cdot \partial_x u = 0 \,.
\end{equation}
Then if \eqref{a 0 Hs}, \eqref{piccib} are fullfilled, for $(\omega, \zeta) \in {\cal O}_\infty^{2 \gamma}$ (see \eqref{buoneacqueb}), under the change of variable $u = \Psi(\omega t)[v] = v(x + \beta(\omega t, x))$ defined in \eqref{perunpugnodidollarib}, the PDE \eqref{eq traporto dimensione alta} transforms into the equation with constant coefficients  
\begin{equation}\label{trasporto dim alta trasformato}
\partial_t v +  m_\infty(\omega , \zeta) \cdot \partial_x v = 0 \,.
\end{equation}
As a consequence, for any $s \geq 0$, $u_0 \in H^s(\T^d)$ the only solutions of the Cauchy problem 
$$
\begin{cases}
\partial_t u + \big( \zeta + a_0(\omega t, x; \omega, \zeta) \big) \cdot \partial_x u = 0  \\
 u(0, x) = u_0(x)
\end{cases}
$$
satisfies $\| u(t) \|_{H^s_x} \lesssim_s \| u_0\|_{H^s_x}$ for any $t \in \R$. Furthermore, if $a_0 = {\rm even}(\varphi, x)$, then $\beta = {\rm odd}(\varphi, x)$. 

\end{teor}
\begin{proof}
Let $(\omega, \zeta) \in {\cal O}_\infty^{2 \gamma}$.{By direct computation, } under the change of coordinates $u = \Psi(\omega t)[v] = v(x + \beta(\omega t, x))$, the equation \eqref{eq traporto dimensione alta} transforms into the PDE 
$$
\partial_t v + \Psi(\omega t)^{- 1}\Big( \omega \cdot \partial_\varphi \beta + \zeta + a_0 + (a_0+\zeta) \cdot \partial_x \beta   \Big) \cdot \partial_x v = 0 \,.
$$
By applying Proposition \ref{esempio1}, one gets that 
$$
\Psi(\omega t)^{- 1}\Big( \omega \cdot \partial_\varphi \beta + \zeta + a_0 + (a_0+\zeta) \cdot \partial_x \beta   \Big) =  m_\infty(\zeta, \omega)
$$
implying that $v$ solves the equation \eqref{trasporto dim alta trasformato}. Such a PDE with constant coefficients can be integrated explicitly, implying that for any $s \geq 0$, $\| v(t) \|_{H^s_x} = \| v(0)\|_{H^s_x}$. Note that, since $\beta \in C^\infty$, by Lemma \ref{change}, $\Psi(\omega t)^{\pm 1}$ is a bounded linear operator $H^k(\T^d) \to H^k(\T^d)$ for any $k \in \N$ with $\sup_{\varphi \in \T^\nu} \| \Psi(\varphi)^{\pm 1}\|_{{\cal B}(H^k_x)} < + \infty$. By using the classical Riesz-Thorin interpolation Theorem for linear operators one gets that $\Psi(\varphi) \in {\cal B}(H^s_x)$ with $\sup_{\varphi \in \T^\nu} \| \Psi(\varphi)^{\pm 1}\|_{{\cal B}(H^s_x)} < + \infty$ for any $s \geq 0$.   Then given $u_0 \in H^s(\T^d)$, one gets that 
$$
\| u(t) \|_{H^s_x} = \|\Psi(\omega t)[v(t)] \|_{H^s_x} \lesssim_s \| v(t)\|_{H^s_x} \lesssim_s \| v(0)\|_{H^s_x} \lesssim_s \| \Psi(\omega t)^{- 1} [u_0]\|_{H^s_x} \lesssim_s \| u_0\|_{H^s_x}
$$
and this concludes the proof. 
\end{proof}
We remark that, as explained in Proposition \ref{esempio1}, the measure of $\cO_0\setminus\cO^{2\g}_\infty$  is of order $\g$, which guarantees  that the reducibility result holds for a  positive measure set of parameters $(\omega,\zeta)$.
On the other hand we are not able to give a bound on $\Omega_0\setminus\Omega^{2\g}_\infty$ unless we impose some further conditions.
We give an example which we believe is interesting for applications.
\begin{cor}\label{misura 1-d}
Let $d = 1$ and consider the one-dimensional transport equation 
	$$
	\partial_t u + \big( m_0{(\omega)} + a_0(\omega t, x; \omega) \big) \partial_x u = 0, \quad x \in \T\,. 
	$$
	Assume that $m_0(\omega)$ satisfies
	\[
	\inf_{\omega\in \Omega_0}|m_0| \ge c \,,\quad |m_0(\omega)|^{lip,\Omega_0}\le C  |m_0(\omega)|^{sup,\Omega_0}
	\]
	for  some appropriate $c,C$ which depend on the set $\Omega_0$.
	Assume finally that \eqref{pasta col pomodoro}, \eqref{piccic} hold.
	\\
Then the analog of Theorem \ref{riducibilita trasporto} hold  for any value of the frequency $\omega$ in the set $\Omega_\infty^{2 \gamma}$ defined in \eqref{buoneacquec}. 
	Moreover if  $\eta_\star,C$ are small enough, depending on the set $\Omega_0$, then the set $\Omega_\infty^{2 \gamma}$ satisfies the measure estimate $|\Omega_0 \setminus \Omega_\infty^{2 \gamma}| \lesssim \gamma$. 
\end{cor}
Note that this includes the case $m_0\neq 0$ and constant in $\omega$. 
\begin{proof}
The fact that the analog of Theorem \ref{riducibilita trasporto} holds is just a repetition of the proof of that statement. The only non trivial thing is to show the measure estimate. 	
We first note that
\[
\inf_{\omega\in \Omega_0}|m_\infty| \ge c -\eta_\star \,,\quad |m_\infty|^{lip,\Omega_0}\le  |m_0|^{lip,\Omega_0}+\eta_\star \le C |m_0|^{sup,\Omega_0}+\eta_\star \le C|m_\infty|^{sup,\Omega_0}+ (1+C)\eta_\star
\]	 Fix $\ell,j$ and compute the measure of a  resonant set
\[
\mathcal R_{\ell j}:=\Big\{\omega\in \Omega_0\,\,:\,\,\lvert \omega\cdot\ell-m_{\infty} j \rvert\le \frac{2\gamma}{\langle\ell \rangle^{\tau}}\Big\}.
\]
We claim that if $\mathcal R_{\ell,j}\neq \emptyset$ then $\lvert j \rvert\le \mathtt k \lvert \ell \rvert$, where \[
\tk:= \sup_{{\omega}\in\Omega_0}\frac{\lvert \omega \rvert+\gamma}{\lvert m_{\infty} \rvert}\le  (c-\eta_\star)^{-1} (\sup_{{\omega}\in\Omega_0}\lvert \omega \rvert+1).
\] 
Note that,  since $\Omega_0$ is compact, $\tk$ is finite.
To prove our claim we note that
\[
|m_{\infty} j |\le |\omega\cdot \ell|+ \lvert \omega\cdot \ell-m_{\infty} j \rvert \le |\omega||\ell| + \frac{\gamma}{\langle \ell \rangle^{\tau}}\le (|\omega|+\g)|\ell|.
\]

In particular $\mathcal R_{0 j}=\emptyset$ for any $j\in\mathbb{Z}\setminus\{0\}$. 
We claim  that $\lvert \mathcal R_{\ell,j} \rvert\le {C} \gamma/\langle \ell \rangle^{\tau+1}$, for some ${C}>0$, for any $(\ell, j)\in\mathbb{Z}^{\nu}\times \mathbb{Z}\setminus\{(0, 0)\}$. We write
$$
\omega = \frac{\ell}{|\ell|} s + v\,, \quad v \cdot \ell = 0\,,
$$
so that setting
$$
\phi(s) := \omega\cdot \ell-m_{\infty}(\omega) j=  |\ell| s + m_\infty\Big( \frac{\ell}{|\ell|} s + v \Big) j \,.
$$
we have (recall $\lvert j \rvert\le \tk \lvert \ell \rvert$)
\[
\phi(s_1)-\phi(s_2) \geq \lvert \ell \rvert (1-\tk \lvert m_{\infty} \rvert^{lip,\Omega_0})\geq \frac{\lvert \ell \rvert}{2}\,,
\]
provided that 
\[
\tk \lvert m_{\infty} \rvert^{lip,\Omega_0}=\sup_{{\omega}\in\Omega_0}(\lvert \omega\rvert +\gamma)\frac{\lvert m_{\infty} \rvert^{lip,\Omega_0}}{\lvert m_{\infty} \rvert}\le \sup_{{\omega}\in\Omega_0} (\lvert \omega\rvert +\gamma) \frac{C \lvert m_{\infty} \rvert^{sup,\Omega_0}+ (1+C)\eta_\star}{\lvert m_{\infty} \rvert^{sup,\Omega_0}}\le \frac{1}{2}.
\]
This last inequality holds true provided that $c-\eta_\star>0$ and
\[
C (\sup_{{\omega}\in\Omega_0}\lvert \omega\rvert +\gamma)<1/4\,,\quad (\sup_{{\omega}\in\Omega_0}\lvert \omega\rvert +\gamma)\frac{1+C}{c-\eta_\star}\eta_\star<\frac14.
\]
Note that the first equation gives an upper bound for $C$, while the second gives an upper bound on $\eta_\star$. We assume that the desired smallness conditions hold and hence we have our claim.
\[
\lvert \Omega_0\setminus \Omega^{2\gamma}_{\infty} \rvert\le \sum_{\ell\in\mathbb{Z}^{\nu}\setminus\{0\}, \lvert j \rvert\le \tk \lvert \ell \rvert} \lvert \mathcal R_{\ell,j}\rvert \le \tk \g \sum_{ \ell \in  \Z^\nu\setminus\{0\}} \frac{1}{|\ell|^\tau}  \lesssim \g.
\]
since $\tau> \nu+1$.
\end{proof}

We also state the following corollary, concerning the solvability of a forced quasi-periodic transport equation. Such a corollary will be applied in \cite{Montsublin}.

\begin{cor}{(Forced case)}\label{cor2}
Assume the hypotheses of the corollary \ref{cor m 0 omega a 0 omega} and let $f:=f(\varphi, x)$ be some $C^{\infty}(\T^{\nu + d})$ function.\\
For every $\omega\in \Omega_{\infty}^{2 \gamma}$ (see \eqref{buoneacquec}), there exists a $C^{\infty}$ function $b(\varphi, x; \omega)$ and a constant $\mathtt{c} = \mathtt{c}(\omega)$ (depending in a Lipschitz way on the parameter $\omega$) such that
\begin{equation}\label{forzata}
\omega\cdot \partial_{\varphi} b(\varphi, x)+(m_0+a_0(\varphi, x)) \partial_x b(\varphi, x)+f(\varphi, x)=\mathtt{c}, \quad (\varphi, x)\in \T^{\nu+1}.
\end{equation}
Furthermore, there exists a constant $\sigma = \sigma (\tau, \nu) > 0$ such that the following estimates hold: 
\begin{equation}\label{stime per applicazione Rick}
\begin{aligned}
& \| b \|_s^{\gamma, \Omega_\infty^{2 \gamma}} \lesssim_s \gamma^{- 1}\Big( \| f \|_{s + \sigma} + \| a_0\|_{s + \sigma}^{\gamma, \Omega_0} \| f \|_{s_0 + \sigma} \Big)  \,, \quad \forall s \geq s_0 \,, \\
& |\mathtt c|^{\gamma, \Omega_0 } \lesssim  \| f \|^{\gamma, \Omega_0}_{s_0}\,. 
\end{aligned}
\end{equation}
\end{cor}
\begin{proof}
	The equation \eqref{forzata} can be written as
	\begin{equation}\label{grotta2}
	\mathcal{L} b+f=\mathtt{c}
	\end{equation}
	where
	\begin{equation}
	\mathcal{L}:=\omega\cdot\partial_{\varphi}- (m_0+a_0(\varphi, x))\partial_x, \quad 
	\end{equation}
	By Corollary \ref{cor m 0 omega a 0 omega}, in particular by \eqref{zanzur5}, we have that $\mathcal{L}=\Psi^{-1}\mathcal{L}_{\infty} \Psi$ , where 
	\begin{equation}
	\Psi h(\varphi, x):=h(\varphi, x+\beta (\varphi, x)), \quad \mathcal{L}_{\infty}:=\omega\cdot \partial_{\varphi}-m_{\infty} \partial_x.
	\end{equation}
	Then
	\begin{equation}
	\mathcal{L}_{\infty} \Psi b=\Psi \big(\mathtt{c}-f \big).
	\end{equation}
	Using that $\Psi(1)=1$ we get
	\begin{equation}\label{grotta}
	\mathcal{L}_{\infty} \Psi  b=\mathtt{c}-\Psi \big(f \big).
	\end{equation}
	and we choose $\mathtt{c}$ such that
	\begin{equation}\label{def mathtt c}
	\mathtt{c}=\langle \Psi \big(f \big)\rangle_{\T^{\nu+1}}
	\end{equation}
	so that the r. h. s. of the equation \eqref{grotta} has zero average. By the fact that $\Psi$ is bounded from $H^s$ to itself for any $s\geq s_0$ and $f\in C^{\infty}$ then $g:=\mathtt{c}-\Psi \big(f \big)\in C^{\infty}$.\\
	Since $g\in C^{\infty}$ and has zero average then the equation $\mathcal{L}_{\infty}[h]=g$, for any $\omega\in \calO^{2\gamma}_{\infty}$, has a $C^{\infty}$-solution which is given by
	\begin{equation}
	h(\varphi, x):=\mathcal{L}_{\infty}^{-1}[g](\varphi, x)=\sum_{(\ell, j)\neq (0,0)} \frac{g_{\ell j}}{\mathrm{i}(\omega\cdot \ell-m_{\infty} j)}\,e^{\mathrm{i}(\ell\cdot \varphi+j x)}.
	\end{equation}
	Furthermore, using the estimate on $m_\infty$ given in \eqref{prrr}, the following standard estimate holds: 
	\begin{equation}\label{stima cal L infty per Rick}
	\| h \|_{s}^{\gamma, \Omega_\infty^{2 \gamma}} \lesssim_s \gamma^{- 1} \| g \|_{s + 2 \tau + 1}^{\gamma, \Omega_\infty^{2 \gamma}}\,, \quad \forall s \geq 0\,. 
	\end{equation}
	Therefore the function
	\begin{equation}\label{def b}
	b:=\Psi^{-1}\mathcal{L}_{\infty}^{-1}[\mathtt{c}-\Psi \big(f \big)]
	\end{equation}
	is a $C^{\infty}$-solution of the equation \eqref{grotta2}.
	
	\noindent
	Finally, the estimates \eqref{stime per applicazione Rick} follow by \eqref{def mathtt c}, \eqref{def b} and by applying the estimates \eqref{A19}, \eqref{larana}, the smallness condition \eqref{piccic} and the estimates \eqref{prrr}, \eqref{stima cal L infty per Rick}. 
\end{proof}
\section{An iterative KAM scheme}\label{quattro}
We prove Theorem \ref{moserMic} and the first item of the Corollary \ref{cor1}.
This is done by applying recursively a KAM step which we  now describe.

\subsection{KAM step}

\noindent
Consider for $\xi \in \cO\subseteq \cO_0$ a Lipschitz family of vector fields  on $\T^{N}$
\begin{equation}\label{mao2}
\begin{aligned}
& 	X:=	(\alpha   + f(\theta;\xi)) \cdot \frac{\partial}{\partial \theta} \\
&\quad |\alpha|^{lip, \cO_0}\le M<2\,,\quad f(\cdot; \xi)\in H^s(\T^{N},\R^d )\quad  \forall s\geq s_0.
\end{aligned}
\end{equation}
Given  $K\gg 1$ and $\g>0$  assume that for some domain $\cO\subseteq \cO_0$ we have (recall \eqref{tripoli})
\begin{equation}\label{piccolezza}
\g^{-1}K^{2\tau+2s_0+1} \|f\|^{\g,\mathcal O}_{s_0} \le \delta:=\delta(s_0) 
\end{equation} 
for some $\delta$ small enough.
Let 
\begin{equation}\label{mao3}
\begin{aligned}
\cO_+ & \equiv \cC_{K,\cO} := \{\xi\in \cO: \; |\alpha\cdot k|>\frac{\g }{ \langle k \rangle^{\tau}}\,,\;\forall k \in \Z^N \setminus \{0\}, \;\; |k|\leq K \}\,,
\end{aligned}
\end{equation}
and for all $\xi \in \cO_{+}$ set $g(\theta ; \xi )$ to be
\begin{equation}\label{perqualchedollaroinpiu}
g(\theta ; \xi ):= \sum_{|k|\leq K} g_{k} e^{\im k\cdot \theta}\,,
\end{equation}
where
\begin{equation}\label{coefficiente fourier g}
g_{k}=
 -\dfrac{f_{k}}{\im \alpha\cdot k}\,, 
 \quad \forall k \in \Z^N \,, \quad 0 < |k| \leq K \,. 
\end{equation}

\begin{lem}\label{lemma KAM step}
The function $g$ defined in \eqref{perqualchedollaroinpiu} satisfies 
\begin{equation}\label{money}
\lVert g \rVert_{s}^{{\g },\mathcal O_{+}} \lesssim {\gamma}^{-1}  \lVert\Pi_K f \rVert^{\g ,\mathcal O}_{s+2\tau+1},\quad \forall \; s\geq s_0.
\end{equation}
The map
\begin{equation}\label{mappaPhi}
\Phi: \theta \mapsto \theta+g(\theta)
\end{equation}
is a diffeomorphism of $\T^{N}$. We have that the pushforward of the vector field $X$ in \eqref{mao2} under the map $\Phi$ in \eqref{mappaPhi} has the form
\begin{equation}\label{mao4}
\Phi_* X:=  	(\alpha_+   + f_+(\theta;\xi)) \cdot \frac{\partial}{\partial \theta}
\end{equation}
where   $\alpha_+ \in \R^N$  is defined and Lipschitz  for $\xi \in\cO_0$, 
 the function $f_+$ is defined and Lipschitz for all $\xi \in \mathcal O_+$ (see \eqref{mao3})
and the following bounds hold:
\begin{equation}\label{nomentana}
|\alpha-\alpha_+|^{\gamma ,\mathcal O_0} \lesssim \lVert f \rVert_{s_0}^{{\g },\mathcal O},
\end{equation}
	\begin{equation}\label{mao0}
	\begin{aligned}
&	\lVert f_+ \rVert_{s_0}^{{\g },\mathcal O_+} \lesssim K^{s_0-s_1}\| f\|_{s_1}^{\g ,\cO} +C_{s_0} \g^{-1}K^{2\tau+2} (\|f\|^{\g ,\cO}_{s_0})^2, \\
& 	\lVert f_+ \rVert_{s}^{{\g },\mathcal O_+} \le \lVert f \rVert^{\g ,\mathcal O}_{s}+ C_s {\gamma}^{-1} K^{2\tau+2 s_0 + 1}\lVert f \rVert^{\g ,\mathcal O}_{s_0} \lVert f \rVert^{\g ,\mathcal O}_{s}. 
	\end{aligned}	
	\end{equation}
Moreover if $f$ is a reversible vector field, then $\theta \mapsto \theta + g(\theta)$ is a reversibility preserving map, implying that $f_+$ is a reversible vector field.	

\noindent	
Let $\lambda_1, \lambda_2 \in B_E$, $\xi  \in \mathcal O_+ (\lambda_1) \cap \calO_+ (\lambda_2)$, $s_1, \mathtt b > 0$ satisfying 
\begin{equation}\label{limitazione indici s2}
s_0 <   \mathtt b + 2 \tau + 3 s_0  + 2 < s_1 \,.
\end{equation}
There exists $\delta':= \delta'(s_1)$ such that if 
\begin{equation}\label{ansatz KAM step variazioni}
 \gamma^{- 1}\big(\| f (\lambda_1)\|_{s_1} + \| f(\lambda_2)\|_{s_1} \big) \leq \delta'
\end{equation}
then  for $\omega\in \cO_+(\lambda_1)\cap \calO_+(\lambda_2)$
\begin{equation}\label{monkey}
\begin{aligned}
&\lVert \Delta_{12} g   \rVert_{s }\lesssim_{s} \gamma^{-1} \Big(\lVert  \Pi_K\Delta_{12} f  \rVert_{s +\tau}+\gamma^{-1}\lvert \Delta_{12} \alpha  \rvert \lVert \Pi_K f \rVert_{s +2\tau+1}\Big), \quad \forall s \geq 0\\
\end{aligned}
\end{equation}
\begin{equation}\label{nomentana2}
\begin{aligned}
& |\Delta_{12} (\al_+ - \al)| \leq \|  \Delta_{12} f \|_{s_0}\,, \\
\end{aligned}
\end{equation}
\begin{equation}\label{tordo}
\begin{aligned}
\lVert \Delta_{12} f_+   \rVert_{s_0 - 1} &\lesssim_{s_0, \tb}  K^{- 1- \mathtt b}  \| \Delta_{12} f\|_{s_0 + \mathtt b} + K^{ \tau + s_0} \gamma^{- 1}\| \Delta_{12} f \|_{s_0 - 1} {\cal M}_{s_0}(f, \lambda_1, \lambda_2) \\
& \qquad + K^{2 \tau + s_0}\gamma^{- 2} |\Delta_{12} \alpha | {\cal M}_{s_0}(f, \lambda_1, \lambda_2)^2,\\
\|\Delta_{12} f_+ \|_{s_0 +\mathtt b} & \lesssim_{s_0, \mathtt b} K^{2 \tau + s_0 }\Big(  \| \Delta_{12} f\|_{s_0 + \mathtt b} + |\Delta_{12} \alpha| \Big) \,. 
\end{aligned}
\end{equation}  
\end{lem}

\begin{proof}
By definition of $\|\cdot \|_s$ (see \eqref{space1}), \eqref{perqualchedollaroinpiu} and \eqref{mao3} we have $\lVert g \rVert_{s} \lesssim \gamma^{-1} K^\tau \lVert a \rVert_{s}$ for all $\xi\in \cO_{+}$. By \eqref{perqualchedollaroinpiu} we have, for $\xi,\xi'\in \calO_{+}$,
\begin{equation}
\begin{aligned}
& |\Delta_{\xi, \xi'} g_{k}|\le  \frac{|\Delta_{\xi, \xi'} f_{k}|}{|\alpha(\xi)\cdot k|}+ \frac{|f_k(\xi')||\Delta_{\xi, \xi'} \alpha | |k|}{|\alpha(\xi)\cdot k||\alpha(\xi')\cdot k|}
\end{aligned}
\end{equation}
hence by \eqref{mao3} we get \eqref{money} and, by using smoothing properties of the projector $\Pi_K$ (see Lemma \ref{lemma interpolazione smoothing}, item $(iv)$) we have
\begin{equation}\label{mao5}
\begin{aligned}
 &  \lVert g \rVert_{s}^{\g,\cO_+} \lesssim \g^{-1} K^{2\tau+1} \lVert  f \rVert_{s}^{\g,\cO}.
 \end{aligned}
\end{equation}

We claim that $g$ satisfies the hypotheses of Lemma \ref{change}, hence $\Phi$ is a diffeomorphism. Indeed, since $s_0\geq [N/2]+3> \gots_0+1$, by \eqref{mao5} and \eqref{piccolezza} we have 
\[
|g|_{1,\infty}^{\gamma, \calO_{+}}\lesssim_{s_0} \|g\|^{\g,\cO_{+}}_{s_0}\lesssim_{s_0} \gamma^{-1} K^{2\tau+1}\|f\|^{\g,\cO}_{s_0}\le \frac12 
\]
provided that $\delta$ in \eqref{piccolezza} is sufficiently small.
By applying Lemma \ref{change}, and by the Sobolev embedding, one gets that the inverse diffeomorphism $y \mapsto y + \widetilde g(\varphi, y)$ satisfies the estimate
\begin{equation}\label{stima widetilde alpha}
\begin{aligned}
\| \widetilde g \|_s^{\gamma, \calO_+} & \lesssim_s  \| g\|_{s +s_0 }^{\gamma, \calO_+}\,,\quad s\ge 0. 
\end{aligned}
\end{equation}
By definition of pushforward (we rename the new variables as $\theta$)
$$
	\Phi_* X:= (\Phi)^{-1}\big(\alpha + f + (\alpha+ f)\cdot \partial_\theta g \big) \cdot \frac{\partial}{\partial \theta}
$$
and by the definition of $g$ in \eqref{perqualchedollaroinpiu}
$$
\begin{aligned}
& \alpha + f + (\alpha+ f)\cdot \partial_\theta g\\
& =  \alpha + \langle f\rangle + \Pi_K^\perp  f +  f\cdot \partial_\theta g.
\end{aligned}
$$
Now we extend $\langle f \rangle$ from $\cO$ to the whole $\cO_0$ by Kirtzbraun theorem, preserving the Lipschitz norm.

\noindent
 We  denote the extension by $\langle f \rangle^{\rm Ext}$ and set
 \begin{equation}\label{aPiu}
 \begin{aligned}
 \alpha_+:= \alpha + \langle f \rangle^{\rm Ext}\,, \quad &\xi\in\cO_0,\\
f_+(\theta):= \Phi^{-1}\Big(\Pi_K^\perp  f +  f\cdot \partial_\theta g\Big), \quad &\xi\in\cO_+.
\end{aligned}
 \end{equation}
The bounds \eqref{nomentana} follow since 
\[
|\langle f \rangle^{\rm Ext} |^{\gamma ,\cO_0}\le |\langle f \rangle |^{\gamma ,\cO}\le \|f\|^{\gamma ,\cO}_{s_0}\,.
\]
 As for \eqref{mao0}  
we repeatedly use  Lemma \ref{change}, 
 indeed setting 
 \begin{equation}\label{effe}
 F:= \Pi_K^\perp  f +  f\cdot \partial_\theta g
 \end{equation}
we have by \eqref{larana}   {for $s\in\mathbb{N}$, $s\geq s_0$}
\begin{equation}\label{maramao}
\begin{aligned}
&\|f_+\|_{s}^{\g,\cO_+}\le \|F\|_{s}^{\g,\cO_+} + C_s\big(\|F\|_{s}^{\g,\cO_+}\|\widetilde g\|^{\g, \cO_+}_{s_0+1}+  \|\widetilde g\|^{\g,\cO_+}_{s+s_0}\|F\|^{\g,\cO_+}_{s_0}\big) \\
& \stackrel{\eqref{stima widetilde alpha}}{\leq} \|F\|_{s}^{\g,\cO_+} + C_s \big( \|F\|_{s}^{\g,\cO_+}\| g\|^{\g, \cO_+}_{ 2 s_0+1}+  \| g\|^{\g,\cO_+}_{s+2 s_0}\|F\|^{\g,\cO_+}_{s_0} \big)\\
&\|F\|_{s}^{\g,\cO_+}\le \|\Pi_K^\perp f(\varphi,x)\|_{s}^{\g,\cO} +C_s( \|f\|^{\g,\cO}_{s_0}\|g\|_{s+1 }^{\g,\cO_+}+ \|f\|_{s}^{\g,\cO}\| g\|_{s_0+1 }^{\g,\cO_+}).
\end{aligned}
\end{equation}
Then if $s=s_0$ by applying the smoothing estimates \eqref{smoothingest} in the second inequality in \eqref{maramao} we get
\begin{equation}\label{mao6}
\begin{aligned}
\|F\|_{s_0}^{\g,\cO_+} &\le K^{s_0-s_1}\| f\|_{s_1}^{\g,\cO} +C_{s_0}  K \|f\|_{s_0}^{\g,\cO}\|g\|_{s_0}^{\g, \calO_+} \\
&\stackrel{(\ref{mao5})}{\le}  K^{s_0-s_1}\|f\|_{s_1}^{\g,\cO}+ \g^{-1}C_{s_0}K^{2\tau+2} (\|f\|_{s_0}^{\g,\cO} )^2,\\[2mm]
\|f_+\|_{s_0}^{\g,\cO_+} &\le (K^{s_0-s_1}\| f\|_{s_1}^{\g,\cO} +C_{s_0} \g^{-1}K^{2\tau+2} (\|f\|^{\g,\cO}_{s_0})^2)(1+ C_{s_0}\g^{-1}K^{2\tau+ 2 s_0 + 1}\|f\|^{\g,\cO}_{s_0} ).
\end{aligned}
\end{equation}
If $s>s_0$ by \eqref{maramao} and \eqref{mao5} we just get
\begin{equation}\label{mao7}
\|f_+\|_{s}^{\g,\cO_+}\le  \|F\|_{s}^{\g,\cO_+}(1+ C_s\g^{-1}K^{2\tau+ s_0 + 2} \|f\|^{\g, \calO}_{s_0}) +C_s  \g^{-1}K^{2\tau+ 2 s_0 + 1} \|f\|^{\g,\cO}_{s}\|F\|^{\g,\cO_+}_{s_0}
\end{equation}
with 
\begin{equation}\label{mao8}
\|F\|_{s}^{\g,\cO_+} {\le} \|\Pi_K^{\perp} f\|_{s}^{\g,\cO} +2C_{s}\g^{-1}K^{2\tau+2}\|f\|_{s_0}^{\g,\cO}\| f\|_{s}^{\g,\cO}\,,
\end{equation}
by using  \eqref{piccolezza} 
we get  \eqref{mao0}.

\noindent
If the vector field $f$ is reversible, then $f = {\rm even}(\theta)$, therefore by the formulae \eqref{perqualchedollaroinpiu}, \eqref{coefficiente fourier g}, one has that $g = {\rm odd}(\theta)$, implying that the diffeomorphism $\theta \mapsto \theta + g(\theta)$ is reversibility preserving. It then follows that the function $F = {\rm even}(\theta)$, by \eqref{effe}. By \eqref{aPiu}, one has that $f_+ = \Phi^{- 1} F$, hence 
$$
f_+(- \theta) = F(- \theta + g(- \theta)) = F\Big(- \big(\theta + g( \theta) \big) \Big) = F(\theta + g(\theta)) = f_+(\theta)
$$
implying that $f_+ = {\rm even}(\theta)$. This proves that $f_+$ is a reversible vector field.

\medskip

\noindent
Now let $\mathtt b$ satisfy \eqref{limitazione indici s2} and $\lambda_1, \lambda_2 \subseteq B_E$ where $B_E \subseteq E$ is a ball in the Banach space $E$. We introduce the notation 
$$
{\cal M}_s (f, \lambda_1, \lambda_2) := {\rm max}\{ \| f(\lambda_1)\|_s\,,\, \| f(\lambda_2)\|_s \}\,, \quad \forall s \geq 0\,. 
$$
The bound in \eqref{monkey} follows since
\[
|\Delta_{12} g _{k}| \leq \frac{|\Delta_{12} f_k|}{|\alpha(\lambda_1)\cdot k|}+\frac{|f_{k}(\lambda_2)| |\Delta_{12} \alpha | |k|}{|\alpha(\lambda_1)\cdot k||\alpha(\lambda_2)\cdot k|} \le
\g^{-1}\langle k\rangle^\tau |\Delta_{12} f_k| +\g^{-2} \langle k\rangle^{2\tau+1}|f_{k}(\lambda_2)||\Delta_{12} \alpha |
\]
for all $\xi\in \calO_+(\lambda_1)\cap\calO_+(\lambda_2)$. Now under the same hypotheses
\begin{equation}\label{dollariNograzie}
\begin{aligned}
\Delta_{12} \alpha_+  &=\Delta_{12} \alpha  +\langle \Delta_{12} f   \rangle, \\
\Delta_{12} f_+  &= (\Delta_{12} \Phi^{- 1})\big[ \Pi_K^{\perp} f(\lambda_1) + f(\lambda_1) \cdot \partial_\theta g(\lambda_1)  \big]  \\
& + \Phi^{- 1}(\lambda_2) \big[ \Pi_K^\bot \Delta_{12} f + (\Delta_{12}f) \cdot \partial_\theta g(\lambda_1) + f(\lambda_2) \cdot \partial_\theta (\Delta_{12} g) \big]\,. 
\end{aligned}
\end{equation}

By using the mean value Theorem, by applying Lemma \ref{change} and using the estimate \eqref{stima widetilde alpha}, for any $s \in [s_0 - 1, s_0 + \mathtt b]$, one gets 
\begin{equation}\label{pappalardoMic}
\begin{aligned}
\| \Delta_{12} \Phi^{- 1}[u] \|_{s} & \lesssim_s \| u \|_{s + 1} \Big(1 + \| g (\lambda_1)\|_{s + 2 s_0} + \| g (\lambda_2)\|_{s + 2 s_0} \Big) \| \Delta_{12} g \|_{s+s_0} \\
& \stackrel{\eqref{money}}{\lesssim_s}  \|  u \|_{s + 1} \Big(1 + C_{s_0}  \gamma^{- 1} {\cal M}_{s + 2 s_0+\tau}(f, \lambda_1, \lambda_2) \Big) \| \Delta_{12} g \|_{s+s_0}\,  \\
& \stackrel{s_2 + \mathtt b + 2 s_0 \leq s_1}{\lesssim_s}  \|  u \|_{s + 1} \Big(1 + C_{s_0}  \gamma^{- 1}{\cal M}_{s_1}(f, \lambda_1, \lambda_2) \Big) \| \Delta_{12} g \|_{s+s_0}\\
& \stackrel{\eqref{ansatz KAM step variazioni}}{\lesssim_s}   \| u\|_{s + 1} \| \Delta_{12} g \|_{s+s_0}. 
\end{aligned}
\end{equation}
Using \eqref{ansatz KAM step variazioni}, applying the estimate \eqref{monkey} and Lemma \ref{lemma interpolazione smoothing}-$(iv)$, one gets 
{\begin{equation}\label{stima delta 12 phi - 1 s2}
\| \Delta_{12} \Phi^{- 1}[u] \|_{s}  \lesssim_s \gamma^{-1} \| u \|_{s + 1}  \Big(\lVert  \Pi_K\Delta_{12} f  \rVert_{s+s_0 +\tau}+\gamma^{-1}\lvert \Delta_{12} \alpha  \rvert \lVert \Pi_K f \rVert_{s +s_0+2\tau+1}\Big)\,, 
\end{equation} }
for $ s \in [s_0 - 1 , s_0 + \mathtt b]$.
Then we have
{
\begin{align*}
\|u\|_{s+1}	=\|\Pi_K^{\perp} f(\lambda_1) + f(\lambda_1)\cdot \partial_\theta  g(\lambda_1)  \|_{s+1} & \leq   \| f(\lambda_1)\|_{s+1} \big( 1 + C \| g (\lambda_1)\|_{s + 2} \big)  \nonumber\\
	& \stackrel{\eqref{mao5}}{\le}  \| f(\lambda_1)\|_{s+1} \Big( 1+C  \gamma^{- 1} \| \Pi_K f(\lambda_1) \|_{s+\tau+2} \Big) \nonumber\\
	& \stackrel{\eqref{ansatz KAM step variazioni}}{\lesssim_s}  \| f(\lambda_1) \|_{s+1}
	\end{align*}
 }
taking $\delta'$ in \eqref{ansatz KAM step variazioni} small enough.
The above estimates imply that 
{\begin{align*}
& \| (\Delta_{12} \Phi^{- 1})\big[ \Pi_K^{\perp} f(\lambda_1) + f(\lambda_1)\cdot \partial_\theta  g (\lambda_1)  \big] \|_{s}  \nonumber\\
& \lesssim_s\gamma^{-1} \| f \|_{s + 1}  \Big(\lVert  \Pi_K\Delta_{12} f  \rVert_{s+s_0 +\tau}+\gamma^{-1}\lvert \Delta_{12} \alpha  \rvert \lVert \Pi_K f \rVert_{s +s_0+2\tau+1}\Big).
\end{align*}}
Specializing the above estimate for $s = s_0 - 1$ and $s = s_0 + \mathtt b$, one gets 
{\begin{align}
& \| (\Delta_{12} \Phi^{- 1})\big[ \Pi_K^{\perp} f(\lambda_1) + f(\lambda_1) \cdot \partial_\theta g (\lambda_1)  \big] \|_{s_0 - 1}  \nonumber\\
& \lesssim_{s_0 } \gamma^{-1} K^{ \tau + s_0}  \| \Delta_{12} f \|_{s_0 - 1} {\cal M}_{s_0}(f, \lambda_1, \lambda_2)    +  \gamma^{- 2} K^{2 \tau + s_0  }\lvert \Delta_{12} \alpha  \rvert {\cal M}_{s_0}(f , \lambda_1, \lambda_2)^2\,, \label{berlusconi 0}
\end{align}}
\begin{equation}\label{pappalardo 1}
\begin{aligned}
& \| (\Delta_{12} \Phi^{- 1})\big[ \Pi_K^{\perp} f(\lambda_1) + f(\lambda_1)\cdot \partial_\theta g (\lambda_1)  \big]\|_{s_0 + \mathtt b} \\
&\lesssim_{s_0, \mathtt b}\gamma^{-1} \| f \|_{s_0+\tb + 1}  \Big(\lVert  \Pi_K\Delta_{12} f  \rVert_{2s_0+\tb +\tau}+\gamma^{-1}\lvert \Delta_{12} \alpha  \rvert \lVert \Pi_K f \rVert_{2s_0+\tb+2\tau+1}\Big)\\
& \stackrel{\eqref{money}, 3s_0 + \mathtt b + 2 \tau + 2 \leq s_1, \eqref{ansatz KAM step variazioni}}{\lesssim_{s_0, \mathtt b}} K^{ \tau + s_0 } \Big(  \lVert  \Delta_{12} f  \rVert_{s_0 + \mathtt b} +  |\Delta_{12} \alpha | \Big) \,.
\end{aligned}
\end{equation}
Furthermore, by Lemma \ref{change}, 
for any $s \in [s_0, s_0 + \mathtt b]$
	\begin{equation}
\begin{aligned}
\| \widetilde g (\lambda_2)\|_{s + s_0} & \stackrel{\eqref{stima widetilde alpha}}{\lesssim_s}  \| g (\lambda_2)\|_{s + 2 s_0} \stackrel{\eqref{money}}{\lesssim_s}   \gamma^{- 1} \| \Pi_K f (\lambda_2) \|_{s + 2 s_0+\tau}  \\
& \stackrel{ \mathtt b + 3 s_0+ \tau \leq s_1}{\lesssim_s}   \gamma^{- 1} \| f(\lambda_2) \|_{s_1} \stackrel{\eqref{ansatz KAM step variazioni}}{\lesssim_s} 1
\end{aligned}
\end{equation}
provided that $\delta'$ in \eqref{ansatz KAM step variazioni} is small enough. 
{ In this way we have, for any $s \in [s_0, s_0 + \mathtt b]$,
	 \[
\|\Phi^{- 1}[u]\|_{s}\lesssim_s \|u\|_{s} (1+ \|\tilde g\|_{s+s_0}) \lesssim_s \|u\|_s,
\]
and consequently
\begin{align*}
	& \| \Phi^{- 1}(\lambda_2) \big[ \Pi_K^\bot \Delta_{12} f + (\Delta_{12} f) \cdot \partial_\theta g (\lambda_1) + f(\lambda_2) \cdot \partial_\theta (\Delta_{12} g ) \big] \|_{s} \\
	& \lesssim_{s} \| \Pi_K^\bot \Delta_{12} f \|_{s} +  \| \Delta_{12} f \|_{s} \| g (\lambda_1) \|_{s+1} + \| f(\lambda_2) \|_{s} \| \Delta_{12} g  \|_{s+1}   \\
	&\lesssim_s \| \Pi_K^\bot \Delta_{12} f \|_{s} +  \g^{-1}\| \Delta_{12} f \|_{s} \| \Pi_K f (\lambda_1) \|_{s+1+\tau} \\ &+ \gamma^{-1}\| f(\lambda_2) \|_{s}  \Big(\lVert  \Pi_K\Delta_{12} f  \rVert_{s +\tau+1}+\gamma^{-1}\lvert \Delta_{12} \alpha  \rvert \lVert \Pi_K f \rVert_{s +2\tau+2}\Big).
	\end{align*}}
Then, using also Lemma \ref{lemma interpolazione smoothing}-$(iv)$, we have
{for $s=s_0-1$
\begin{align}
& \| \Phi^{- 1}(\lambda_2) \big[ \Pi_K^\bot \Delta_{12} f + (\Delta_{12} f) \cdot \partial_\theta g (\lambda_1) + f(\lambda_2) \cdot \partial_\theta (\Delta_{12} g ) \big] \|_{s_0 - 1} \nonumber\\
& \stackrel{\eqref{lemma interpolazione smoothing}-(iv)}{\lesssim_{s_0, \mathtt b}} K^{- 1- \mathtt b}  \| \Delta_{12} f\|_{s_0 + \mathtt b} + K^{\tau} \gamma^{- 1}\| \Delta_{12} f \|_{s_0 - 1} {\cal M}_{s_0}(f, \lambda_1, \lambda_2) \nonumber\\
& \qquad + K^{2 \tau }\gamma^{- 2} |\Delta_{12} \alpha | {\cal M}_{s_0}(f, \lambda_1, \lambda_2)^2, \label{pappalardo 3-a}
\end{align}
similarly for $s=s_0+\tb$.
\begin{align}
& \| \Phi^{- 1}(\lambda_2) \big[ \Pi_K^\bot \Delta_{12} f + (\Delta_{12} f) \cdot \partial_\theta g(\lambda_1) + f(\lambda_2) \cdot \partial_\theta (\Delta_{12} g )\big] \|_{s_0 + \mathtt b} \nonumber\\
& \stackrel{\eqref{money}, \eqref{monkey}, s_0 + \mathtt b + 2\tau + 2 \leq s_1}{\lesssim_{s_0, \mathtt b}}\| \Pi_K^\bot \Delta_{12} f \|_{s_0 + \mathtt b} +  K^\tau  \gamma^{- 1}\| \Delta_{12} f \|_{s_0 + \mathtt b} {\cal M}_{s_1}(f, \lambda_1, \lambda_2) \nonumber\\
& \qquad  + K^{2 \tau + 1}\gamma^{- 2} |\Delta_{12} \alpha | {\cal M}_{s_1}(f, \lambda_1, \lambda_2)^2   \stackrel{\eqref{ansatz KAM step variazioni}}{\lesssim_{s_0, \mathtt b}} K^{2\tau+1} \big( \| \Delta_{12} f\|_{s_0 + \mathtt b} + |\Delta_{12} \alpha | \big) \,. \label{pappalardo 3}
\end{align}
}
The estimate \eqref{tordo} then follows by recalling \eqref{dollariNograzie} and by applying the estimates \eqref{pappalardo 1}, \eqref{pappalardo 3}.

\end{proof}
\subsection{KAM iteration}
Now we describe the iteration of the KAM step.
\begin{lem}\label{ittero}
Consider the vector field $X_0$ in \eqref{mao1}.
Recall \eqref{tripoli} and set
\begin{equation}\label{tordo3}
\begin{aligned}
	& \chi:= \frac32 \,,\quad  \mu > 4  \tau + 2 s_0 + 4  ,\quad \varrho > 2 \tau + 2 s_0 + 1  ,\quad  s_1 > \chi \mu + s_0 , \\
	&  \kappa > 8 \tau + 2 s_0 + 4\,, \quad    \tb> \mu \chi + \kappa + 1 .
	\end{aligned}
\end{equation}
	There exists $K_0$ depending on $s_0, \nu$ and $\delta_*:=\delta_*(s_1)$ small such that if
	\begin{equation}\label{small}
	\delta_0(s_1) K_0^{\varrho}\le \delta_* , \qquad \mbox{where} \qquad\delta_0(s_1):=\gamma^{-1} \lVert f_0 \rVert_{s_1}^{\gamma, \cO_0}
	\end{equation}
	then, for all $n\ge 0$, the following holds. We set $K_n:=K_0^{\chi^n}$, $\chi:=3/2$ and
	\begin{equation}\label{buoneacque2}
	\mathcal{O}_{n+1}= \mathcal C_{K_n,\mathcal O_n}:=\left\{ \xi \in \mathcal{O}_n : \lvert \alpha_n(\xi) \cdot k \rvert\geq  \frac{{\gamma}}{ \langle k  \rangle^{\tau}},  \quad \forall k \in \mathbb{Z}^{N} \setminus\{ 0 \}, \quad | k |\le K_n  \right\}
	\end{equation}
	and for all $\xi \in \mathcal{O}_{n+1}$ we set $g _{n+1}(\theta ; \xi ) $ to be 
	\begin{equation}\label{defbeta}
	\begin{aligned}
	g_{n+1}( \theta  ; \xi ) & := \sum_{|k |\le K_n}g^{(n+1)}_{k} e^{\im k \cdot \theta }\,, \\
	 g^{(n+1)}_{k}& := \frac{a^{(n)}_{k}}{\mathtt i \alpha_n \cdot k} \,, \quad k \in \Z^{N} \setminus \{ 0\}\,, \quad |k| \leq K_n
	\end{aligned}
\end{equation}
and
\begin{equation}\label{colomba}
\delta_n(s):=\gamma^{-1} \lVert f_n \rVert_{s}^{\gamma, \cO_n}.
\end{equation}	

\smallskip

Moreover we set 
\begin{equation}\label{lambda M}
\lambda:=\lambda(s):=1/(s-s_0+1), \quad \mathtt{M}(s):=\max\{\delta_{0}(s_1),\delta_{0}(s)\}.
\end{equation}

Then the following holds.

\medskip

\noindent$(\mathcal{P}_1)_{n}$. Set $g_0=0$.  For all $n\ge 0$ the torus diffeomorphism $\Phi_{n}: \theta \mapsto \theta + g_n(\theta)$ is well defined and the induced operator \eqref{grotta3} acts on $H^s$ to itself $\forall s \geq s_0$. Setting
\begin{equation}\label{zanzur}
X_{n}:= (\Phi_{n})_*X_{n-1}:=  \big( \alpha_n(\xi) + f_n(\theta; \xi) \big) \cdot \frac{\partial}{\partial \theta} 
\end{equation}
we have the bounds 
\begin{align}\label{proto}
&	|\alpha_{n}-\alpha_{n-1}|^{\g}\lesssim   \g \delta_0(s_1)K_{n}^{-\mu} K_0^\mu\,, \quad |\alpha_n|^{lip}\le M_0+C \gamma^{-1}\delta_0(s_1) 
\end{align}
and there exists a positive constant $C_1(s)$ such that
\begin{align}
\label{colomba1}
&\delta_{n}(s_0)\le \delta_0(s_1) \,K_0^{\mu}\,K_{n}^{-\mu},\quad 
\delta_{n}(s)\le C_1(s)\, \delta_0(s)(1 + \sum_{j=1}^n 2^{-j})\,,\quad s\geq s_0 .
\end{align}
As a consequence 
\begin{align}
\label{colomba 1 2}
& \delta_n(s) \leq  C (s)\, K_{n}^{-\la \mu} K_0^{\la\mu}\mathtt{M}(s+1)\,,  \\
\label{pagliuzza2}
&\| g_{n}\|_s^{\g, \cO_{n }} \le \delta_n(s+2\tau+1)\le C_2(s)\,  K_{n}^{-\la \mu} K_0^{\lambda\mu}\,\mathtt{M}(s + 2 \tau + 2), \quad s\geq s_0
	\end{align}
	for some $C_2(s)>0$.

\medskip

\noindent $(\mathcal{P}_2)_{n}$. The torus diffeomorphism defined by 
\begin{equation}\label{zanzur2}
\begin{cases}
\Psi_0= \mathrm{I},\\
\Psi_n= \Phi_n\circ\Psi_{n-1}
\end{cases}
\end{equation}
 is of the form 	$\Psi_{n}: \theta\mapsto \theta+h_n(\theta)$ with, for all $s\geq s_0$, (recall \eqref{pagliuzza2} for the definition of $\mathtt{M}(s)$)
\begin{equation}\label{mao8ancora}
\|h_n\|_{s_0}^{\g, \cO_{n }} \le C(s_0) \delta_0({s_1}) \sum_{j=0}^n 2^{-j}\,,\quad \|h_n\|^{\g, \cO_{n }}_s \le C_3(s)\, \mathtt{M}(s+2 \tau + s_0 + 2)\,\sum_{j=0}^{n} 2^{-j} \,,
\end{equation}
\begin{equation}\label{misonorotta}
\|h_{n-1} -h_n\|^{\g, \cO_{n }}_{s_0}\le C(s_0) \delta_0(s_1)2^{-n}  
\,,\quad 
\|h_{n-1}-h_n\|^{\g,  \cO_{n }}_s \le C_4(s)\, \mathtt{M}(s+2 \tau + s_0 + 3) 2^{-n}.
\end{equation}
Moreover, if $f_0$ is a reversible vector fields, then $\theta \mapsto \theta + g_n(\theta)$, $\theta \mapsto \theta + h_n(\theta)$ are reversibility preserving maps and $f_n$ is a reversible vector field. 

\medskip

\noindent $(\mathcal{P}_3)_{n}$. Let $\lambda_1, \lambda_2 \in B_E$.
There exists a constant $C_*(s_1) > 0$ and $\tilde{\delta}:=\tilde{\delta}(s_1)$ such that if
\begin{equation}\label{condizione di piccolezza iterazione variazioni}
K_0^{2 \tau +s_0 +\chi\mu}   \gamma^{- 1} \big( \| f_0(\lambda_1) \|_{s_1} + \| f_0(\lambda_2) \|_{s_1 }  \big) \leq \tilde{\delta}
\end{equation}
then for any $n \geq 0$, for any $\xi \in {\cal O}_{n }(\lambda_1) \cap {\cal O}_{n }(\lambda_2)$, the following estimates hold:  
\begin{align}
 \| \Delta_{12} f_n\|_{s_0 - 1}& \leq C_*(s_1)  K_n^{- \mu} \| \Delta_{12} f_0\|_{s_0 + \mathtt b} \,, \label{stima Delta 12 an bassa}\\
 \| \Delta_{12} f_n\|_{s_0 + \mathtt b}& \leq C_*(s_1)K_n^{\kappa} \| \Delta_{12} f_0\|_{s_0 + \mathtt b} \,, \label{stima Delta 12 an alta} \\
 |\Delta_{12}(\alpha_{n+1} - \alpha_{n })|& \leq   \| \Delta_{12} f_{n }\|_{s_0 - 1} , \label{stima delta 12 m n m n - 1} \\
 |\Delta_{12} \alpha_n| & \lesssim  \| \Delta_{12} f_0 \|_{s_0 + \mathtt b}, \label{Delta 12 m n} \\
 \| \Delta_{12} h_n\|_{s_0 - 1} & \leq C_*(s_1) \g^{-1} \sum_{j = 0}^{n} 2^{- j} \| \Delta_{12} f_0\|_{s_0  + \mathtt b}\,.  \label{stima Delta 12 beta n}
\end{align}

\medskip

\noindent $(\mathcal{P}_4)_{n}$. Let $\lambda_1, \lambda_2 \in B_E$, $0 < \gamma - \rho < \gamma < 1$ satisfy
\begin{equation}\label{piccolezza delta 12 a0 S4}
K_{n - 1}^{\tau + 1} \| \Delta_{12} f_0\|_{s_0 + \mathtt b} \leq \rho\,. 
\end{equation} 
 Then ${\cal O}_n^\gamma(\lambda_1) \subseteq {\cal O}_n^{\gamma - \rho}(\lambda_2)$. 
\end{lem}

\begin{proof}
The statements $(\mathcal{P}_{1,2})_0$ are trivial. $(\mathcal{P}_3)_0$ follows taking $C_*(s_1)$ large enough, for instance $C_*(s_1)> K_0^{\mu}$. The statement $(\mathcal{P}_3)_0$ holds by setting ${\cal O}_0^{\gamma}(\lambda_1) = {\cal O}_0 = {\cal O}_0^{\gamma - \rho}(\lambda_2)$.

\noindent 
Now suppose that $(\mathcal{P}_{1, 2})_n$ hold and we prove that $(\mathcal{P}_{1, 2})_{n+1}$ also hold. 
	
\medskip	
	
\noindent	\textbf{Proof of $(\mathcal{P}_{1})_{n+1}$.}
We have to prove that the $(n+1)$-th   diffeomorphism of the torus is well defined from $H^{s}$ to itself for all $s \geq s_0$. In particular, we show that \eqref{piccolezza} holds with  $K\rightsquigarrow K_n$ and $f\rightsquigarrow f_n$.\\
We have
\begin{equation}\label{ice}
\delta_n(s_0) K_n^{2\tau+2 s_0+1}\le \delta_0(s_1)\, K_n^{2\tau+2 s_0+1-\mu} K_0^{\mu}.
\end{equation}
By \eqref{tordo3}
	$\mu>2\tau+2 s_0+1$. Hence $K_n^{2\tau+2 s_0+1-\mu}$ is a decreasing sequence and by \eqref{small}, \eqref{tordo3} ($\rho>2\tau+2 s_0+1$) 
\begin{equation}\label{mao8'}
\delta_0(s_1)\, K_n^{2\tau+2 s_0+1-\mu} K_0^{\mu}\le  \delta_0(s_1) K_0^{2\tau+2 s_0+1}\le \delta_*.
\end{equation}
Then by \eqref{ice}, \eqref{mao8'} and taking $\delta_*\le \delta$ (recall \eqref{small} and \eqref{piccolezza}) we get our first claim
\[
 \delta_n(s_0) K_n^{2\tau+2 s_0+1}\le \delta.
\]
In order to prove \eqref{colomba1} we apply the KAM step with $f_+\rightsquigarrow f_{n+1}$ .\\
We start by estimating the low norm. By \eqref{mao6}
\begin{equation}\label{mao9}
\begin{aligned}
\delta_{n+1}(s_0)	&\le \g^{-1}(K_n^{s_0-s_1}\| f_n\|_{s_1}^{\g,\cO_{n}} +C_{s_0}\g^{-1} K_n^{2\tau+2} (\|f_n\|^{\g,\cO_n}_{s_0})^2)(1+C_{s_0} \g^{-1}K_n^{2\tau+2 s_0+1}\|f_n\|^{\g,\cO_n}_{s_0} )\\
&
\le
(K_n^{s_0-s_1}\delta_n(s_1) +C_{s_0} K_n^{2\tau+2} (\delta_n(s_0)^2)(1+C_{s_0} K_n^{2\tau+ 2 s_0+1}\delta_n(s_0)\ ).
\end{aligned}
\end{equation}
We first note that $C_{s_0} K_n^{2\tau+ 2 s_0+1-\mu}\,K_0^\mu\,\delta_0(s_1)<1$, indeed, since $\mu> 2\tau+2 s_0+1$, this is a decreasing sequence and by \eqref{small} (taking $\delta_*$ small enough) and $\varrho>2\tau+2 s_0+1$ (see \eqref{tordo3})
\[
C_{s_0} K_0^{2\tau+2 s_0}\,\delta_0(s_1)\le C_{s_0} \delta_*<1.
\]

Hence (recall \eqref{lambda M})
\[
\delta_{n+1}(s_0)\le 2 K_n^{s_0-s_1}\delta_n(s_1) +2 C_{s_0} K_n^{2\tau+2} \delta_n(s_0)^2\le \delta_0(s_1)\,K_{n+1}^{-\lambda \mu} K_0^{\lambda \mu} 
\]
provided that
	\begin{equation}
	\begin{cases}
	2\delta_n(s_1) K_n^{-(s_1-s_0)}<\frac{1}{2} \delta_0(s_1) \,K_0^{\mu}\,K_{n+1}^{-\mu},\\
	2\,C_{s_0} \,\delta_n(s_0)^2 K_n^{2\tau+2} <\frac{1}{2} \delta_0(s_1) \,K_0^{\mu}\,K_{n+1}^{-\mu}.
	\end{cases}
	\end{equation}

	Thus, by the inductive hypotesis \eqref{colomba1}, we have to prove 
	\begin{equation}\label{mao9'}
		8 C_{1}(s_1) K_n^{-(s_1-s_0)+\chi\mu} K_0^{-\mu}<1\,,\quad 	4 C_{s_0} \delta_0(s_1) K_0^{\mu} K_n^{2\tau+2-(2-\chi)\mu}<1.
	\end{equation}
	By \eqref{tordo3} we have
	\begin{equation}\label{pane2}
	s_1-s_0>\chi\mu\,,\quad \mu > \frac{2\tau+2}{2-\chi}=4\tau+4\stackrel{(\ref{tripoli})}{=}4\nu+12,
	\end{equation}
	then the sequences in \eqref{mao9'} are decreasing and we just need
	\[
	8 C_1(s_1) K_0^{-(s_1-s_0)+\mu (\chi-1)}<1\,,\quad 4 C_{s_0} \delta_0(s_1)  K_0^{2\tau+2+\mu(\chi-1)}<1,
	\]
which follows by taking  $K_0$ sufficiently large (depending on $C_{s_0}$ and $C_1(s_1)$ in \eqref{colomba1}) and by \eqref{small}, since
	\[
\varrho>2\tau+ 	2+\mu(\chi-1)=2\tau+2+\frac{\mu}{2}.
	\]

Regarding the estimates in high norm, by \eqref{mao0} we have for all $s> s_0$
\begin{equation}\label{bled3}
	\lVert f_{n+1} \rVert_{s}^{{\g},\mathcal O_{n+1}} \le \lVert f_n \rVert^{\g,\mathcal O_n}_{s}+  C_s  {\gamma}^{-1} K_n^{2\tau+1+2 s_0}\lVert f_n \rVert^{\g,\mathcal O_n}_{s_0} \lVert f_n \rVert^{\g,\mathcal O_n}_{s}.
\end{equation}
First we prove the following bounds for $s>s_0$
\begin{align}
& \delta_n(s) \le (\mathtt{C}(s))^n \delta_0(s) (1+\sum_{j=1}^n 2^{-j}), \quad 0\le n \le n_0(s), \label{bled}\\ 
& \delta_n(s) \le (\mathtt{C}(s))^{n_0(s)} \delta_0(s) (1+\sum_{j=1}^n 2^{-j}), \quad n\geq n_0(s),\label{bled2}
\end{align}
for some constant $\mathtt{C}(s)$ and for a suitable $n_0(s)$. For $n=0$ \eqref{bled} holds by taking $\mathtt{C}(s)\geq 1$. For $n\le n_0(s)-1$ we have, by \eqref{bled3},
\begin{align*}
\delta_{n+1}(s) &\le \delta_n(s)(1+C_s \gamma K_n^{2\tau+2 s_0+1}\delta_n(s_0))\\
&\le (\mathtt{C}(s) )^n \delta_0(s) (1+\sum_{j=1}^n 2^{-j}) (1+C_s \gamma K_n^{2\tau+2 s_0+1}\delta_n(s_0))\\
&\le (\mathtt{C}(s))^{n+1} \delta_0(s)  (1+\sum_{j=1}^{n+1} 2^{-j}) 
\end{align*}
provide that
\begin{equation}\label{bled4}
\frac{C_s}{\mathtt{C}(s)}\gamma K_n^{2\tau+2 s_0+1-\mu} K_0^{\mu}\delta_0(s_1)\le \frac{2^{-(n+1)}}{1+\sum_{j=1}^n 2^{-j}}.
\end{equation}
Considering that $n\le n_0(s)-1$, by \eqref{tordo3} and \eqref{small} with $\delta_*$ small enough, we get \eqref{bled4}.\\
Now consider $n\geq n_0(s)$. By \eqref{bled3} we have
\begin{align*}
\delta_{n+1}(s) &\le \delta_n(s)(1+C_s \gamma K_n^{2\tau+2 s_0+1}\delta_n(s_0))\\
&\le (\mathtt{C}(s) )^{n_0(s)} \delta_0(s) (1+\sum_{j=1}^n 2^{-j}) (1+C_s \gamma K_n^{2\tau+2 s_0+1}\delta_n(s_0))\\
&\le (\mathtt{C}(s))^{n_0(s)} \delta_0(s)  (1+\sum_{j=1}^{n+1} 2^{-j}) 
\end{align*}
provide that
\begin{equation}\label{bled5}
C_s \gamma K_n^{2\tau+2 s_0+1-\mu} K_0^{\mu}\delta_0(s_1)\le \frac{2^{-(n+1)}}{1+\sum_{j=1}^n 2^{-j}}.
\end{equation}
The bound \eqref{bled5} follows by \eqref{tordo3}, \eqref{small} and by choosing $n_0(s)$ large enough. Hence we proved the second estimate in \eqref{colomba1} by setting
\[
C_1(s):=(\mathtt{C}(s))^{n_0(s)}.
\]

 
\medskip 
 
Now we prove \eqref{pagliuzza2}. For $s\geq s_0 $,  setting $\lambda=1/(s-s_0+1)$, we have
\begin{equation}\label{A1b}
\|f_{n}\|^{\g, \calO_{n}}_{s}\leq (\| f_{n}\|_{s_0}^{\g, \calO_{n}})^{\lambda}(\| f_{n}\|^{\g, \calO_{n}}_{s+1})^{1-\lambda},  
\end{equation}
from which we may deduce that (recall \eqref{colomba1})
\begin{equation}\label{mao10}
\delta_{n}(s) \le (K_{n}^{- \mu} K_0^{\mu}\delta_{0}(s_1))^\lambda (\delta_{n}(s+1))^{1-\lambda}\le 2\, C_1(s+1)\, K_{n}^{-\lambda \mu} K_0^{\lambda\mu}\mathtt{M}(s+1).
\end{equation}
By \eqref{mao10}
\begin{equation}\label{tse0}
\| g_{n+1}\|^{\g, \calO_{n+1}}_{s}\leq \delta_{n}(s+2\tau+1) \le2\, C_1(s+2\tau+2)\, K_{n}^{-\lambda \mu} K_0^{\lambda\mu}\mathtt{M}(s+2\tau+2) \,,\quad s\geq s_0
\end{equation}
which is \eqref{pagliuzza2} taking $C_2(s)\geq 2\, C_1(s+2\tau+2)$.
The bounds \eqref{colomba1} trivially implies \eqref{proto}.

\medskip

\noindent \textbf{Proof of $(\mathcal{P}_2)_{n+1}$.}
By construction 
\begin{equation}\label{formula iterativa beta n}
h_{n+1}(\theta)= g_{n+1}(\theta) +h_n(\theta+ g_{n+1}(\theta))
\end{equation}
thus, by \eqref{larana},   {for $s\in\mathbb{N}$, $s\geq s_0$ }
\begin{equation}\label{pane}
\begin{aligned}
\|h_{n+1}\|_{s}^{\g,\cO_{n+1}} &\le \| g_{n+1}\|^{\g, \calO_{n+1}}_s + \|h_n\|^{\g, \calO_{n}}_{s}(1+C_s\| g_{n+1}\|^{\g, \calO_{n+1}}_{s_0+1})\\
& + C_s \|h_n\|^{\g, \calO_{n}}_{s_0}\| g_{n+1}\|^{\g, \calO_{n+1}}_{s+s_0}.
\end{aligned}
\end{equation}
First we show the following. By fixing an opportune $n_0(s)\in\mathbb{N}$, we have the bounds
\begin{align}
& \lVert h_n \rVert_s\le (\mathtt{C}(s))^n \mathtt{M}(s+2\tau+s_0+2)\sum_{j=0}^{n} 2^{-j} \quad 0\le n\le n_0(s),\label{prima}\\ 
& \lVert h_n \rVert_s\le (\mathtt{C}(s))^{n_0(s)} \mathtt{M}(s+2\tau+s_0+2)\sum_{j=0}^{n} 2^{-j} \quad n \geq n_0(s).\label{seconda}
\end{align}

We recall that $h_0:=\alpha_0=0$. By \eqref{tse0}, \eqref{mao8ancora}, \eqref{pane} we have for $n\le n_0(s)-1$
\begin{align*}
&\|h_{n+1}\|_{s}^{\g,\cO_n}\le C_2(s)\,K_{n}^{-\la \mu} K_0^{\la\mu}\,\mathtt{M}(s+2\tau+2+s_0)\,(1+C_sC(s_0)\delta_{0}(s_1)\sum_{j=0}^n 2^{-j})  
\\ &
+(\mathtt{C}(s))^n\,\mathtt{M}(s+2\tau+s_0+2)\,\sum_{j=0}^n 2^{-j}(1+C_s K_{n}^{2\tau+2-\mu}K_0^\mu \delta_0(s_1)) \le
\\ &
\le(\mathtt{C}(s))^{n+1}\,\mathtt{M}(s+2\tau+s_0+2)\sum_{j=0}^{n+1} 2^{-j}
\end{align*}
provided that we choose $\mathtt{C}(s)$ such that (recall \eqref{small})
\begin{equation}\label{ultrabook}
 (\mathtt{C}(s))^n\geq 2 K_n^{-\lambda\mu} C_s C(s_0) \delta_0(s_1), \quad \mathtt{C}(s)\geq \max\{ 2(1+C_s), C_2(s) K_0^{\lambda\mu}\}.
\end{equation}
Hence we proved \eqref{prima}. Now, by \eqref{tse0}, \eqref{mao8ancora}, \eqref{pane}, we have for $n\geq n_0(s)$
\begin{align*}
&\|h_{n+1}\|_{s}^{\g,\cO_n}\le C_2(s)\,K_{n}^{-\la \mu} K_0^{\la\mu}\,\mathtt{M}(s+2\tau+2+s_0)\,(1+C_sC(s_0)\delta_{0}(s_1)\sum_{j=0}^n 2^{-j})  
\\ &
+(\mathtt{C}(s))^{n_0(s)}\,\mathtt{M}(s+2\tau+s_0+2)\,\sum_{j=0}^n 2^{-j}(1+C_s K_{n}^{2\tau+2-\mu}K_0^\mu \delta_0(s_1)) \le
\\ &
\le(\mathtt{C}(s))^{n_0(s)}\,\mathtt{M}(s+2\tau+s_0+2)\sum_{j=0}^{n+1} 2^{-j}
\end{align*}
provide that we choose $n_0(s)$ large enough so that (recall $\mu>2\tau+2$)
\[
C_s K_n^{2\tau+2-\mu} K_0^{\mu}\delta_0(s_1)\le \frac{2^{-(n+1)}}{1+\sum_{j=1}^n 2^{-j}}
\]
and by using \eqref{ultrabook} with $n=n_0(s)$.\\
Hence we proved \eqref{mao8ancora} with
\[
C_3(s):=\max\{(\mathtt{C}(s))^{n_0(s)}, C(s_0)\}.
\]
In order to prove the first bound in \eqref{misonorotta} we use \eqref{tse0}, \eqref{mao8ancora}, \eqref{A20b} and we have
\begin{equation}\label{laptop}
\|h_{n+1}-h_n\|^{\g, \calO_{n+1}}_{s_0}\le \lVert g_{n+1} \rVert_{2 s_0}(1+C_{s_0} \lVert h_n \rVert_{s_0+1}^{\gamma, \calO_n}).
\end{equation}
By interpolation we get
\[
\lVert h_n \rVert^{\gamma, \calO_n}_{s_0+1}\le (\lVert h_n \rVert^{\gamma, \calO_n}_{s_0})^{1/2}(\lVert h_n \rVert^{\gamma, \calO_n}_{s_0+2})^{1/2}\stackrel{(\ref{mao8ancora})}{\le} \tilde{C}(s_0) \delta_0(s_1) \sum_{j=0}^n 2^{-j}\le 2 \tilde{C}(s_0) \delta_0(s_1) \stackrel{(\ref{small})}{\le} 1
\]
where
\[
\tilde{C}(s_0):=(C(s_0))^{1/2} (C_3(s_0+2))^{1/2}.
\]
Hence we have by \eqref{laptop} 
\[
\begin{aligned}
\|h_{n+1}-h_n\|^{\g, \calO_{n+1}}_{s_0}\le 2 \lVert g_{n+1} \rVert^{\gamma, \calO_{n+1}}_{2 s_0}\stackrel{(\ref{money})}{\le} K_n^{2\tau+1+s_0} \delta_n(s_0)&\stackrel{(\ref{colomba1})}{\le} \delta_0(s_1) K_n^{2\tau+1+s_0-\mu} K_0^{\mu}\\
&\le C(s_0) \delta_0(s_1) 2^{-(n+1)}
\end{aligned}
\]
provided that $C(s_0)> K_0^{\mu}$, $\mu>2\tau+1+s_0$ and $K_0>1$ is large enough. Now we prove the second bound in \eqref{misonorotta}.\\
By \eqref{tse0}, \eqref{mao8ancora}, \eqref{A20b}, \eqref{pagliuzza2} we have 
\begin{align*}
	\|h_{n+1}-h_n\|^{\g, \calO_{n+1}}_{s} &\le \| g_{n+1}\|^{\g, \calO_{n+1}}_s + C(s) (\|h_n\|^{\g, \calO_{n}}_{s+1}\| g_{n+1}\|^{\g, \calO_{n+1}}_{s_0} + \|h_n\|^{\g, \calO_{n}}_{s_0}\| g_{n+1}\|^{\g, \calO_{n+1}}_{s+s_0}) \\
& \le  C_4(s)\,\mathtt{M}(s+2\tau+s_0+3)\, 2^{-(n+1)}
\end{align*}
provided that $C_4(s)$ is large enough
and
\[
K_n^{-\lambda\mu} K_0^{\lambda\mu}\le 2^{-(n+1)}, \quad K_n^{2\tau+1}\delta_n(s_0)\le 2^{-(n+3)}, \quad 2\delta_0(s_1) K_n^{-\lambda\mu} K_0^{\lambda\mu}\le 1
\]
which hold by taking $K_0>1$ large enough, by \eqref{colomba1} and \eqref{small}.

\noindent
If $f_n$ is a reversible vector field, by Lemma \ref{lemma KAM step} (recall also the definitions \eqref{defbeta})one has that $\theta \mapsto \theta + g_{n + 1}(\theta)$ is a reversibility preserving map and $f_{n + 1} $ is a reversible vector field. Furthermore, since by the inductive hypotheses, $\theta \mapsto \theta + h_n (\theta)$ is a reversibility preserving map, by the formula \eqref{formula iterativa beta n} one immediately gets that $\theta \mapsto \theta + h_{n + 1}(\theta) $ is reversibility preserving too. 

\medskip

\noindent
\textbf{Proof of $(\mathcal{P}_3)_{n+1}$.} 
If we take (recall $\delta_*$ in \eqref{small}, \eqref{tordo3})
$$\tilde{\delta}\le K_0^{-\varrho+2\tau+1}\delta_*,$$
since $f_0(\lambda_1)$ and $f_0(\lambda_2)$ satisfy the smallness assumption \eqref{condizione di piccolezza iterazione variazioni}, then condition \eqref{small} holds for both $f_0(\lambda_1)$ and $f_0(\lambda_2)$ and we can apply the estimates proved in the steps $(\mathcal{P}_1)_{n}$, $(\mathcal{P}_2)_{n}$ obtaining that
\begin{equation}\label{hart schwanz}
\begin{aligned}
& \| f_n(\lambda_1) \|_{s_1}, \| f_n(\lambda_2)\|_{s_1} & \stackrel{(\ref{colomba1})}{\lesssim_{s_1}}  {\cal M}_{s_1 }(f_0, \lambda_1, \lambda_2) .
\end{aligned}
\end{equation}
This estimate implies that \eqref{ansatz KAM step variazioni} is verified by \eqref{condizione di piccolezza iterazione variazioni}. 
Then by applying \eqref{tordo}, one gets that 

\begin{align}
	\| \Delta_{12} f_{n + 1}\|_{s_0 - 1} & \lesssim_{s_0} K_n^{- 1 - \mathtt b}  \| \Delta_{12} f_n \|_{s_0 + \mathtt b} + K_n^{ \tau + s_0 } \| \Delta_{12} f_n \|_{s_0 - 1} {\cal M}_{s_0}(f_n , \lambda_1, \lambda_2) \\
	& \qquad + K_n^{2  \tau + s_0 }\gamma^{- 2} |\Delta_{12} \alpha_n | {\cal M}_{s_0}(f_n, \lambda_1, \lambda_2)^2 \nonumber\\
	& \stackrel{\eqref{colomba1}, \eqref{stima delta 12 m n m n - 1}, \eqref{condizione di piccolezza iterazione variazioni}, \eqref{stima Delta 12 an bassa}, \eqref{stima Delta 12 an alta}}{\lesssim_{s_0, \mathtt b}} K_n^{\kappa - 1 - \mathtt b} \| \Delta_{12} f_0\|_{s_0 + \mathtt b}  \nonumber\\
	& \qquad \qquad + C_*(s_1)K_n^{ 2\tau + s_0 - 2 \mu}  \| \Delta_{12} f_0\|_{s_0 + \mathtt b} K_0^{2\mu}\delta_0(s_1) \nonumber\\
	& \leq C_*(s_1) K_{n + 1}^{- \mu }\| \Delta_{12} f_0\|_{s_0 + \mathtt b}
	\end{align} 
%
provided for any $n \geq 0$,
$$
C(s_0, \mathtt b) K_{n + 1}^{ \mu} K_n^{- \mathtt b -1 + \kappa } \leq \frac{C_*(s_1)}{2}\,, \quad C(s_0, \mathtt b) K_{n + 1}^{ \mu}K_n^{2 \tau + s_0 - 2  \mu} K_0^{2\mu} \delta_0(s_1)\leq \frac{C_*(s_1)}{2}\,.  
$$
As in the previous items the left hand side of these inequalities is decreasing in $n$, since by \eqref{tordo3} we have $\mu > \frac{2 \tau + s_0}{(2 - \chi)}$ , $\mathtt b >  \mu \chi + \kappa  - 1$. Then our claim follows by taking $K_0, C_*(s_1) > 0$ large enough. 
Moreover 
\begin{align*}
\| \Delta_{12} f_{n + 1} \|_{s_0 + \mathtt b} & \lesssim_{s_0, \mathtt b} K_n^{2 \tau + s_0 }\Big(  \| \Delta_{12} f_n \|_{s_0 + \mathtt b} + |\Delta_{12} \alpha_n | \Big)   \\
& \stackrel{\eqref{stima Delta 12 an alta}, \eqref{Delta 12 m n}}{\lesssim_{s_0, \mathtt b}} K_n^{\kappa + 2 \tau + s_0 } \| \Delta_{12} f_0\|_{s_0 + \mathtt b} \leq K_{n + 1}^{\kappa} \| \Delta_{12} f_0 \|_{s_0 + \mathtt b}
\end{align*}
provided $C(s_0, \tb) K_n^{2 \tau + s_0  + \kappa} \leq K_{n + 1}^{\kappa}$ for any $n \geq 0$. Such a condition is fullfilled, by taking $K_0 > 1$ large enough, since by \eqref{tordo3} one has that $(\chi - 1)\kappa > 2 \tau + s_0 $. Therefore, the estimates \eqref{stima Delta 12 an bassa}, \eqref{stima Delta 12 an alta} have been proved at the step $n + 1$. The estimates \eqref{stima delta 12 m n m n - 1}, \eqref{Delta 12 m n} follow by applying \eqref{nomentana2} by using a telescoping argument. 

\noindent
The estimates \eqref{money}, \eqref{mao8ancora}, using that $ 2 \tau + 2 s_0 + 3 \leq s_1$ imply that 
\begin{equation}\label{stima alpha n lambda 1 2}
\begin{aligned}
\| g_n(\lambda_1)\|_{2 s_0}, \| g_n (\lambda_2) \|_{2 s_0} & \lesssim_{s_0} K^{ \tau  + s_0}_n \gamma^{- 1} {\cal M}_{s_0}(f_n, \lambda_1, \lambda_2)  \\
& \lesssim_{s_0}  K_n^{\tau+ s_0 -  \mu} K_0^\mu \delta_0(s_1) \stackrel{\eqref{condizione di piccolezza iterazione variazioni}}{\leq} 1\,, \\
\| h_n(\lambda_1)\|_{s_0 }, \| h_n(\lambda_2)\|_{s_0 } & \lesssim_{s_0} {\cal M}_{s_1}(a_0, \lambda_1, \lambda_2)\,. 
\end{aligned}
\end{equation}
By \eqref{monkey}, \eqref{colomba1}, \eqref{Delta 12 m n} and \eqref{stima Delta 12 an bassa}, \eqref{condizione di piccolezza iterazione variazioni}, one gets the estimate 
\begin{equation}\label{stima Delta 12 alpha n}
\| \Delta_{12} g_n\|_{s_0 - 1} \lesssim_{s_0} \gamma^{- 1}\| \Delta_{12} f_0\|_{s_0 + \mathtt b} K_n^{2 \tau + 1 - \mu}\,.
\end{equation}
By the formula \eqref{formula iterativa beta n} one gets 
$$
\begin{aligned}
h_{n + 1}(\theta ;  \lambda_1) - h_{n + 1}(\theta ; \lambda_2) & = g_{n+1}(\theta ; \lambda_1) -  g_{n+1}(\theta ; \lambda_2)  \\
& \quad +h_n( \theta+ g_{n+1}(\theta; \lambda_1) ; \lambda_1) - h_n(\theta + g_{n+1}(\theta; \lambda_2) ; \lambda_2)\\
& = \Delta_{12} g_{n + 1}(\theta) + (\Delta_{12} h_n)( \theta  + g_{n + 1}(\theta; \lambda_1)) \\
& \qquad + h_n( \theta  + g_{n + 1}( \theta ; \lambda_1); \lambda_2) - h_n( \theta  + g_{n + 1}( \theta ; \lambda_2); \lambda_2)\,.
\end{aligned}
$$
Using the triangular inequality, the mean value theorem, the estimates \eqref{stima alpha n lambda 1 2}, \eqref{stima Delta 12 alpha n}, Lemma \ref{change} and the smallness condition \eqref{condizione di piccolezza iterazione variazioni}, one gets the estimate 
\begin{equation}
\begin{aligned}
\| \Delta_{12} h_{n + 1}\|_{s_0 - 1} &\leq C_{s_0} \gamma^{- 1} \| \Delta_{12} f_0\|_{s_0 + \mathtt b} K_n^{2 \tau + 1 -  \mu} \\
&+ \| \Delta_{12} h_n\|_{s_0 - 1}\Big(1 + C_{s_0} K_n^{2 \tau + 1 + s_0 - \mu} \gamma^{- 1}{\cal M}_{s_1}(f_0, \lambda_1, \lambda_2) \Big)\,.
\end{aligned}
\end{equation}
Then using the induction hypothesis \eqref{stima Delta 12 beta n}, one gets 
\begin{equation}
\begin{aligned}
\| \Delta_{12} h_{n + 1}\|_{s_0 - 1} & \leq C_{s_0} \gamma^{- 1} \| \Delta_{12} f_0\|_{s_0 + \mathtt b} K_n^{2 \tau + 1 -  \mu}  + C_*(s_1) \sum_{j = 0}^{n} 2^{- j} \| \Delta_{12} f_0\|_{s_0 + \mathtt b} \gamma^{- 1} \\
& \quad + C_*(s_1) C_{s_0} K_n^{2 \tau + 1 + s_0 -  \mu} \gamma^{- 1}{\cal M}_{s_1}(f_0, \lambda_1, \lambda_2) \sum_{j = 0}^{n} 2^{- j} \| \Delta_{12} f_0\|_{s_0 + \mathtt b} \gamma^{- 1} \\
& \leq C_*(s_1) \sum_{j = 0}^{n + 1} 2^{- j} \| \Delta_{12} f_0\|_{s_0 + \mathtt b} \gamma^{- 1}
\end{aligned}
\end{equation}
provided 
\[
C_{s_0} K_n^{2\tau+1-\mu}\le C_*(s_1) \sum_{j=0}^{n+1} 2^{-j}, \qquad C_{s_0}\gamma^{-1} K_n^{2\tau+1+s_0-\mu} \mathcal{M}_{s_1}(f_0, \lambda_1, \lambda_2)\le 1.
\]
This condition is fullfilled, by \eqref{condizione di piccolezza iterazione variazioni}, taking $K_0$ and $C_*(s_1)$ large enough, recalling that $K_n = K_0^{\chi^n}$ for any $n \geq 0$ and since $2 \tau + 1 + s_0 - \mu < 0$. 

\noindent
Finally, we prove the statement $({\cal P}4)_{n + 1}$. Let $\xi \in {\cal O}_{n + 1}^\gamma(\lambda_1)$. By the definition \eqref{buoneacque2}, $\xi \in {\cal O}_n^\gamma(\lambda_1)$ and the induction hypothesis implies that $\xi   \in {\cal O}_n^{\gamma - \rho}(\lambda_2)$. Since, trivially ${\cal O}_n^{\gamma}(\lambda_1) \subseteq {\cal O}_n^{\gamma - \rho}(\lambda_1)$, one has that 
$$
\xi  \in {\cal O}_{n + 1}^\gamma(\lambda_1) \subseteq {\cal O}_n^{\gamma - \rho}(\lambda_1) \cap {\cal O}_n^{\gamma - \rho}(\lambda_2)\,. 
$$
We can then apply the estimate \eqref{Delta 12 m n} implying that for any $\omega \in {\cal O}_{n + 1}^\gamma(\lambda_1) \subseteq {\cal O}_n^{\gamma - \rho}(\lambda_1) \cap {\cal O}_n^{\gamma - \rho}(\lambda_2)$ one has that 
$$
|\Delta_{12} \alpha_n | \lesssim \| \Delta_{12} f_0\|_{s_0 + \mathtt b}\,.  
$$
Therefore, for any $k \in \Z^N \setminus\{ 0\}$, $|k| \leq K_{n}$, one has that 
\begin{align}
|\alpha_n(\xi; \lambda_2 ) \cdot k |  & \geq |\alpha_n(\xi; \lambda_1) \cdot k | - |\Delta_{12} \alpha_n|| k | \nonumber\\
& \geq \frac{\gamma}{\langle k  \rangle^\tau} - K_n \| \Delta_{12} a_0\|_{s_0 + \mathtt b} \geq \frac{\gamma - \rho}{\langle k \rangle^\tau}
\end{align}
By the condition \eqref{piccolezza delta 12 a0 S4}. Then $\xi  \in {\cal O}_{n + 1}^{\gamma - \rho}(\lambda_2)$, which is the claimed statement. 
\end{proof}

\subsection{Proof of Theorem \ref{moserMic}}

Now we can prove the Theorem \ref{moserMic}.
\begin{proof}
We fix $s_1$ as in \eqref{tordo3} and choose $\eta_\star$ so that \eqref{picci} implies \eqref{small}, namely (recall \eqref{tordo3})
\[
K_0^{\varrho}\eta_*\le \delta_*.
\]
 Consider now the sequence $h_n$ defined in Lemma \ref{ittero}-($\mathcal{P}_2$). By formula \eqref{misonorotta} this is a Cauchy sequence in $H^s(\T^{N})$ for all $s\geq s_0$. Let us denote by $h^{(\infty)}$ its limit. We note that $h^{(\infty)}$ belongs to $\cap_{s\geq s_0} H^s(\T^{N})$, hence it is a $C^{\infty}$ function in  $\theta$. As a consequence $\Psi^{(\infty)}$ is $C^\infty$ torus diffeomorphism.\\
In  the same way, by \eqref{proto} the sequence $\alpha_n$ is a Cauchy sequence and we denote by $\alpha_\infty$ its limit.\\
We claim that 
\begin{equation}\label{zanzur5}
(\Psi^{(\infty)})^{-1}\Big(  \xi + f_0 + (\xi + f_0)\cdot \partial_\theta h^{(\infty)}) \Big)=\alpha_{\infty}.
\end{equation}
First we prove by induction that (recall \eqref{zanzur})
\begin{equation}\label{zanzur3}
(\Psi_n)_* X_0=X_n.
\end{equation}
For $n=0$ this is trivially true. Now prove the $(n+1)$-th step. Recalling the definition \eqref{zanzur2}, by the composition of pushforwards
\[
(\Psi_{n+1})_* X_0=(\Phi_{n+1})_*(\Psi_n)_* X_0=(\Phi_{n+1})_* X_n=X_{n+1}.
\]
Now by \eqref{zanzur3} we have that
\begin{equation}\label{zanzur4}
(\Psi_n)^{-1}\Big( \xi + f_0 + (\xi + f_0) \cdot \partial_\theta  h_n  \Big)= \alpha_n+ f_n.
\end{equation}
By \eqref{colomba1} the r. h. s. of \eqref{zanzur4} converges in $H^{s_0}$ to $\alpha_{\infty}$. By the fact that $h_n$ converges to $h^{(\infty)}$ in $H^s$, for every $s\geq s_0$, then
\[
(\Psi^{(\infty)})^{-1}\Big(  \xi + f_0 + (\xi + f_0)\cdot \partial_\theta h^{(\infty)}) \Big) -(\Psi_n)^{-1}\Big( \xi + f_0 + (\xi + f_0) \cdot \partial_\theta  h_n  \Big)
\]
converges to $0$ in $H^{s_0}$ by using triangle inequalities, the mean value theorem and the bounds given in Lemma \ref{change}. Then we proved our claim.\\
By \eqref{zanzur5}, setting $\Psi^{(\infty)}: \theta \mapsto ( \theta +h^{(\infty)}(\theta ))$, we have
\[
\Psi^{(\infty)}_* X_0 =  \alpha_\infty(\xi) \cdot \frac{\partial}{\partial\theta}\,,\quad \forall \xi \in \cap_n\cO_n.
\]
The bounds \eqref{tordo4} follow by  \eqref{proto}.
In order to complete the proof we need to show that 
\[
\calO^{2\gamma}_{\infty}\subset \bigcap_n \mathcal{O}_n.
\]
We prove this by induction. By definition $\calO^{2\gamma}_{\infty}\subset \mathcal{O}_0$. 
Suppose that $\calO^{2\gamma}_{\infty}\subset \mathcal{O}_n$ and we claim that $\calO^{2\gamma}_{\infty}$ is included in $\mathcal{O}_{n+1}$.\\
Fix $\xi \in\calO^{2\gamma}_{\infty}$ and $\lvert k \rvert\le K_n$. Then by \eqref{proto}, \eqref{small} and recalling $\mu$ in \eqref{tordo3} we have
\[
\lvert \alpha_n(\xi) \cdot k \rvert\geq \lvert \alpha_\infty(\xi) \cdot k \rvert-\lvert \alpha_{\infty}-\alpha_n \rvert\lvert k \rvert\geq \frac{2 \gamma}{\langle k \rangle^{\tau}} - \delta_0(s_1) K_0^{\mu} K_{n-1}^{-\mu} K_n^2 K_n^{\tau} \geq \frac{\gamma}{ \langle k  \rangle^{\tau}}.
\]
Finally, note that if $f_0$ is a reversible vector field, all the diffeomorphisms $\theta \mapsto \theta + h_n(\theta)$ are reversibility preserving, namely $h_n = {\rm odd}(\theta) $ for any $n \in \N$. Hence the limit function $h^{(\infty)} = \lim_{n \to + \infty} h_n$ is ${\rm odd}(\theta)$ implying that the map $\theta \mapsto \theta + h^{(\infty)}(\theta)$ is reversibility preserving.  
The proof of the theorem is then concluded. 
\end{proof}
\begin{remark}
Lemma \eqref{ittero}-$({\cal P}4)_n$ implies that if $0 < \gamma - \rho < \gamma < 1$ and $\lambda_1 , \lambda_2 \in B_E$ satisfies $K_{n - 1}^{\tau + 1} \| \Delta_{12} a_0\|_{s_0 + \mathtt b} \leq \rho$, then ${\cal O}_\infty^{2 \gamma}(\lambda_1) \subseteq \cap_{m = 0}^n {\cal O}_m^{\gamma - \rho}(\lambda_2)$. 
\end{remark}

\appendix

\section{Technical Lemmata}\label{lemmitecnici}

In this Section we present standard tame and Lipschitz estimates for composition of functions and changes of variables.

\medskip
Let us denote $L^{\infty}:=L^{\infty}(\T^{d}, \mathbb{C})$ and $W^{s, \infty}:=W^{s, \infty}(\T^d, \mathbb{C})$ with $d\geq 1$. The norms of these spaces are respectively indicated with $\lvert \cdot \rvert_{L^{\infty}}:=\lvert \cdot \rvert_{0, \infty}$, $\lvert \cdot \rvert_{s, \infty}$ and are defined by
\begin{equation}\label{holderinfinito}
\lvert u \rvert_{L^{\infty}}:=\sup_{x\in \T^d} \lvert u(x) \rvert, \quad
\lvert u \rvert_{s, \infty}:=\sum_{s_1\le s} \lvert D^{s_1} u \rvert_{L^{\infty}}, \quad \lvert D^{s_1} u \rvert_{L^{\infty}}:=\sup_{\lvert \vec{s}_1 \rvert=s_1} \lvert \partial_x^{\vec{s}_1} u \rvert_{L^{\infty}},
\end{equation}
here $D^s$ is the $s$-th Fr\'echet derivative with respect to $x$, hence $D^s$ is a symmetric multi-linear operator.

Let us denote with $H^s:=H^s(\T^d, \mathbb{C})$ the space of Sobolev functions on $\T^d$ defined by
\begin{equation}
H^s(\T^{d}, \mathbb{C}):=\left\{ u\in L^2(\T^{d}) : \lVert u \rVert_s^2:=\sum_{j\in\mathbb{Z}^d} \lvert u_{ j} \rvert^2 \langle j \rangle^{2 s}<\infty \right\}.
\end{equation}
We shall actually use the equivalent norm
\begin{equation}
\lVert u \rVert_s:=\lVert u \rVert_{H^s(\T^d)}:=\lVert u \rVert_{L^2(\T^d)}+\lVert D^s u \rVert_{L^2(\T^d)}, \quad \lVert D^s u \rVert_{L^2(\T^d)}:=\sup_{\lvert \vec{s} \rvert=s} \lVert \partial_x^{\vec{s}} u \rVert_{L^2(\T^d)}.
\end{equation}

%
\begin{lem}\label{lemma interpolazione smoothing}
	Let $s_0>d/2$. Then the following holds.
	\begin{itemize}
		\item[$(i)$] {\bf Embedding}. $\lvert u \rvert_{L^{\infty}}\le \lVert u \rVert_{s_0}$ for all $u\in H^{s_0}$.
		\item[$(ii)$] {\bf Algebra}. $\lVert u v \rVert_{s_0}\le C(s_0) \lVert u \rVert_{s_0} \lVert v \rVert_{s_0}$ for all $u, v\in H^{s_0}$.
		\item[$(iii)$] {\bf Interpolation}. For $0\le s_1\le s\le s_2$, $s=\lambda s_1+(1-\lambda) s_2$, $\lambda\in [0, 1]$,
		\[
		\lVert u \rVert_s\le \lVert u \rVert_{s_1}^{\lambda}\lVert u \rVert_{s_2}^{1-\lambda}, \quad \forall u\in H^{s_2}.
		\]
		Let $a_0, b_0\geq 0$ and $p, q>0$. For all $u\in H^{a_0+p+q}$, $v\in H^{b_0+p+q}$
		\[
		\lVert u \rVert_{a_0+p}\lVert v \rVert_{b_0+q}\le \lVert u \rVert_{a_0+p+q}\lVert v \rVert_{b_0}+\lVert u \rVert_{a_0}\lVert v \rVert_{b_0+p+q}.
		\]
		Similarly
		\[
		\lvert u \rvert_{s, \infty}\le C(s_1, s_2) \lvert u \rvert_{s_1, \infty}^{\lambda} \lvert v \rvert_{s_2, \infty}^{1-\lambda} \quad \forall u\in W^{s_2, \infty}
		\]
		and for all $u\in W^{a_0+p+q}$, $v\in W^{b_0+p+q}$
		\[
		\lvert u \rvert_{a_0+p, \infty}\lvert v \rvert_{b_0+q, \infty}\le C(a_0, b_0, p, q) \lvert u \rvert_{a_0+p+q, \infty}\lvert v \rvert_{b_0, \infty}+\lvert u \rvert_{a_0, \infty}\lvert v \rvert_{b_0+p+q, \infty}.
		\]
		
		
		\item[$(iv)$] For any $s, \alpha \geq 0$, 
		\begin{equation}\label{smoothingest}
		\| \Pi_N u \|_{s + \alpha} \leq N^\alpha \| u \|_s\,, \quad \| \Pi_N^\bot u\|_s \leq N^{- \alpha} \| u \|_{s + \alpha}
		\end{equation}
		where 
		$$
		\Pi_N u(\varphi, x) := \sum_{|(\ell, j)| \leq N} u_{\ell, j} e^{\ii (\ell \cdot \varphi + j x)}\,, \quad \Pi_N^\bot := {\rm Id} - \Pi_N\,. 
		$$
	\end{itemize}
\end{lem}
\begin{remark}
	If $u=u(\omega)$ and $v=v(\omega)$ depend in a Lipschitz way on a parameter $\calO\subset \mathbb{R}^{\nu}$, all the previous statements hold by replacing $\lvert \cdot \rvert_{s, \infty}$, $\lVert \cdot \rVert_s$ with the Lipschitz norms $\lvert \cdot \rvert^{\g,\cO}_{s, \infty}$, $\lVert \cdot \rVert^{\g,\cO}_s$, provided that we take $s_0> d/2+1$ (i.e. all the relations hold with $s_0+1$, for $s_0>d/2$ and then we rename $s_0$).
	Indeed we first apply the formulas above to  the variation $(u(\omega)-u(\omega'))/(\omega-\omega')$, this implies the desired bounds for  the  norm
	$
	\max\{ \lVert u \rVert^{sup}_s, \gamma \lVert u \rVert_{s-1}^{lip}\}
	$.
Since this norm is equivalent to  $\lVert \cdot \rVert_s^{\gamma, \calO}$, our claim follows.
\end{remark}

\begin{lem}\label{change}{\bf (Change of variable)} Consider $p\in W^{s,\infty}({\mathbb T}^{d};{\mathbb R}^{d})$, $s\geq 1$, with $|p|_{1, \infty}\leq1/2$.
	Let $f(x)=x+p(x)$. Then:\\
	$(i)$ $f:\mathbb{T}^d\to \mathbb{T}^d$ is a diffeomorphism, its inverse is $f^{-1}(y)=g(y)=y+q(y)$ with
	$q\in W^{s,\infty}({\mathbb T}^{d};{\mathbb R}^{d})$ and $|q|_{s, \infty}\leq C|p|_{s, \infty}$. More precisely,
	\begin{equation}\label{A18}
	|q|_{L^{\infty}}=|p|_{L^{\infty}}, \; |D q|_{L^{\infty}}\leq 2| D p|_{L^{\infty}},
	\; | D q|_{s-1, \infty}\leq C| D p|_{s-1, \infty},
	\end{equation}
	where the constant $C$ depends on $d,s$ 

	\smallskip	
	
	\noindent	$(ii)$ If $u\in H^{s}({\mathbb T}^{d};{\mathbb C})$, then $u\circ f(x)=u(x+p(x))\in H^{s}({\mathbb T}^{d};{\mathbb C})$, and, with the same $C$ as
	in $(i)$ one has
	\begin{subequations}\label{A20}
		\begin{align}
		\|u\circ f\|_{s}&\leq \|u\|_{s}+ C(\|u\|_s|p|_{1, \infty}+  |Dp|_{s-1, \infty }\|u\|_{1}),\label{A20a}\\
		\|u\circ f-u\|_{s}&\leq C(|p|_{L^{\infty}}\|u\|_{s+1}+|p|_{s, \infty }\|u\|_{2}),\label{A20b}
		\end{align} 
	\end{subequations}
	\smallskip		
	
	\noindent $(iii)$	Assume that $p=p_{{\omega}}$ depends in a Lipschitz way by a parameter 
	${\omega}\in\calO\subset \mathbb{R}^{\nu}$, and suppose, as above, that $| p_{{\omega}}|_{1,\infty}\leq1/2$
	for all ${\omega}$. Then $q=q_{{\omega}}$ is also Lipschitz in ${\omega}$, and
	%
	%
	\begin{equation}\label{A19}
	|q|_{s, \infty}^{\g,\cO}\leq C\left(|p|^{\g,\cO}_{s,\infty}+
	\{\sup_{{\omega}\in\calO}|p_{{\omega}}|_{s, \infty }\}
	|p|_{1,\infty}^{\g,\cO}
	\right)\leq C|p|^{\g,\cO}_{s,\infty},
	\end{equation}
	\begin{equation}\label{larana}
	\|u\circ f\|_s^{\g,\calO} 
	\le \|u\|_{s}^{\g,\calO}+C (\lVert u \rVert_{s}^{\g,\calO}|p|^{\g,\calO}_{1,\infty}+|p|^{\g,\calO}_{s,\infty}\|u\|_{2}^{\gamma,\calO}),\,\,\,\,\,  {s\in\mathbb{N}}
	\end{equation}
	the constant $C$ depends on $d,s$ (it is independent on ${\gamma}$).
	
\end{lem}

\begin{proof} The estimate \eqref{A18} is proved  in \cite{Ono}, \eqref{A20b} is proved in the Appendix of \cite{Airy}. The bounds \eqref{A20a} are slightly different from the corresponding ones of  \cite{Ono}, \cite{Airy}.  This and the different choiche of wheighted Lipschitz norm reflect on  the Lipschitz bounds \eqref{A19} and \eqref{larana}.  Let us prove \eqref{A20a}. We follow the proof of Lemma $B.4$-$(ii)$ in \cite{Ono} abut treat in a different way some terms arising from the Faa di Bruno's formula. 
	First we note that $\| u\circ f\|_0\le  \|u\|_0(1+2|Dp|_{\infty})$.
	Then we consider the expression
	\[
	D^s (u\circ f)=\sum_{k=1}^s \sum_{ j_1+\dots+j_k=s,\,j_i\geq 1} C_{k}\,\, ( D^k u)[D^{j_1} f,\dots, D^{j_k} f]\,.
	\]
	Here the coefficients $C_k$ are integer numbers which take into account the combinatorics, it is easily seen that $C_s=1$ . We and note that  $D f=\mathrm{I}+D p$, while $D^{j} f = D^j p$ for $j\ge 1$. Then we split the sum above in the following way
	
	\begin{align}
	D^s (u\circ f) &=  \,\, ( D^s u)\circ f[D f,\dots, D f]+\sum_{k=1}^{s-1}\,\,\, \sum_{\substack{ j_1+\dots+j_k=s,\\ \nonumber
	 \prod_i j_i>1}} C_{k}\,\, ( D^k u)[D^{j_1} f,\dots, D^{j_k} f]\\ \nonumber
	  & =  (D^s u)\circ f+\sum_{r=1}^{s} \binom{s}{r}\,\, (D^s u)\circ f[\underbrace{Dp, \dots, Dp}_{\times \,r\,
	}, \underbrace{\mathrm{I}, \dots, \mathrm{I}}_{\times\, s-r\,
	}]\\
	\label{yotvata}
	&+\sum_{k=1}^{s-1}\,\,\, \sum_{\substack{ j_1+\dots+j_k=s,\\ \prod_i j_i>1}} C_{k}\,\, ( D^k u)\circ f[D^{j_1} f,\dots, D^{j_k} f]
	\\ \nonumber
	&= (D^s u)\circ f\\ \nonumber
	&+ \sum_{k=1}^{s}\sum_{\substack{r_1,r_2\\ 0<r_1+r_2\leq k}} \sum_{\substack{ j_1+\dots+j_{r_1}= s+r_1-k \\ j_i>1 }}\!\!\!\!\!\!C_{k,r_1,r_2}\,\, ( D^k u)\circ f[D^{j_1} p,\dots, D^{j_{r_1}} p, \underbrace{Dp, \dots, Dp}_{\times\, r_2}, \underbrace{\mathrm{I}, \dots, \mathrm{I}}_{\times\,k-r_1-r_2}]\nonumber
	\end{align}

	The first summand is estimated by noting that 
	\[
	\|(D^s u)\circ f\|_0 \le (1+ 2|Dp|_\infty)\|u\|_s.
	\]
	Now we rename $j_i = h_i+1$ for $i=1,r_1$ in the second summand and set $G=Dp$,  we get
	\[
	\sum_{\substack{ h_1+\dots+h_{r_1}= s-k  }} C_{k,r_1,r_2}\,\, ( D^k u)\circ f[D^{h_1} G,\dots, D^{h_{r_1}} G, \underbrace{G, \dots, G}_{\times\,r_2}, \underbrace{\mathrm{I}, \dots, \mathrm{I}}_{\times\,k-r_1-r_2}].
	\]
	The $L^2$ norms of the summands above are bounded by 
	\[
	2 C_{k,r_1,r_2} \|  u\|_k |G|_{h_1,\infty}\dots |G|_{h_{r_1},\infty}|G|_\infty^{r_2} .
	\]
	Then one can follow exactly the same proof of Lemma $B.4$-$(ii)$ in \cite{Ono}.
	
	\smallskip
	
	In order to prove \eqref{A19} we use formula (6.15) of \cite{Airy} which reads in terms of the Lipschitz seminorm
	\[
	\gamma | q|^{lip,\cO}_{s-1,\infty} \le C\left( |p|^{sup,\cO}_{s-1,\infty} + \gamma |p|_{s-1,\infty}^{lip,\cO}+ |p|^{sup,\cO}_{s,\infty} (|p|^{sup,\cO}_{0,\infty}+\gamma |p|_{0,\infty}^{lip,\cO})\right).
	\]
	In order to prove \eqref{larana} we compute the Lipschitz variation.  
	We have
	\[
	\|u(\omega,x+p_{\omega}(x))-u(\omega', x+p_{\omega'}(x))\|_s= \|\Delta_\omega u \circ f_\omega + (u \circ \widehat f - u )\circ f_{\omega'}\Delta_\omega p\|_s \,,\quad \widehat f= f_{\omega}\circ g_{\omega'}\,,
	\]
	now in the r.h.s. the first summand is  bounded by using  \eqref{A20a};  regarding the second summand we use the interpolation estimates for products and then \eqref{A20a}, \eqref{A20b}.
	Note that \eqref{A20b} {\em loses one derivative}, this  is why in  the norm $\lVert \cdot \rVert^{\g, \calO}$ we require the estimate of the Lipschitz variation of $\omega$ only for the norm $\lVert \cdot \rVert_{s-1}$. 
\end{proof}
\bibliography{bibliografia.bib}
\end{document}